\documentclass[11pt, twoside]{article}

\usepackage{amssymb}
\usepackage{amsmath}
\usepackage{amsthm}
\usepackage{color}
\usepackage{mathrsfs}

\usepackage{indentfirst}

\usepackage{txfonts}

\allowdisplaybreaks

\def\BB{{\mathfrak B}}

\def\vp{p(\cdot)}
\def\lv{{L^{\vp}(\rn)}}

\def\rr{{\mathbb R}}
\def\rn{{\mathbb{R}^n}}

\def\vh{{H_A^{\vp}(\rn)}}
\def\vfah{{H_{A,\,{\rm fin}}^{\vp,r,s}(\rn)}}
\def\vfahfz{{H_{A,\,{\rm fin}}^{\vp,\fz,s}(\rn)}}
\def\zz{{\mathbb Z}}

\def\nn{{\mathbb N}}

\def\ca{{\mathcal A}}

\def\cp{{\mathcal P}}

\def\dfrac{\displaystyle\frac}

\def\cs{{\mathcal S}}

\def\fz{\infty }
\def\az{\alpha}

\def\lz{\lambda}

\def\vaz{\varepsilon}
\def\HL{M_{{\rm HL}}}

\def\lf{\left}
\def\r{\right}

\def\hs{\hspace{0.35cm}}
\def\ls{\lesssim}
\def\gs{\gtrsim}

\def\noz{\nonumber}

\def\com{\complement}

\def\loc{{\mathop\mathrm{\,loc\,}}}
\def\supp{\mathop\mathrm{\,supp\,}}

\def\XXint#1#2#3{{\setbox0=\hbox{$#1{#2#3}{\int}$ }
\vcenter{\hbox{$#2#3$ }}\kern-.6\wd0}}

\DeclareMathOperator{\esssup}{ess\,sup}
\DeclareMathOperator{\essinf}{ess\,inf}

\def\({\left(}
\def \){ \right)}

\def\vp{p(\cdot)}
\def\lv{{L^{\vp}(\rn)}}

\newtheorem{theorem}{Theorem}[section]
\newtheorem{lemma}[theorem]{Lemma}
\newtheorem{corollary}[theorem]{Corollary}

\newtheorem{proposition}[theorem]{Proposition}

\theoremstyle{definition}
\newtheorem{remark}[theorem]{Remark}
\newtheorem{definition}[theorem]{Definition}
\renewcommand{\appendix}{\par
   \setcounter{section}{0}%
   \setcounter{subsection}{0}%
   \setcounter{subsubsection}{0}%
   \gdef\thesection{\@Alph\c@section}%
   \gdef\thesubsection{\@Alph\c@section.\@arabic\c@subsection}%
   \gdef\theHsection{\@Alph\c@section.}%
   \gdef\theHsubsection{\@Alph\c@section.\@arabic\c@subsection}%
   \csname appendixmore\endcsname
 }

\numberwithin{equation}{section}

\begin{document}

\title{\bf\Large
Anisotropic Variable Campanato-Type Spaces and Their Carleson Measure Characterizations
\footnotetext{\hspace{-0.35cm} 2020 {\it
Mathematics Subject Classification}. Primary 42B35;
Secondary 42B30, 46E30, 28C20.
\endgraf {\it Key words and phrases.}
expansive matrix, (variable) Campanato-type
space, (variable) Hardy space, tent space,
duality, Carleson measure.
\endgraf This project is supported
by the National Natural Science Foundation of China (Grant No.
11971125).}}
\author{Long Huang and Xiaofeng Wang\footnote{Corresponding author,
E-mail: \texttt{wxf@gzhu.edu.cn}.
}
}
\date{}
\maketitle

\vspace{-0.7cm}

\begin{center}
\begin{minipage}{13cm}
{\small {\bf Abstract}\quad
Let $p(\cdot):\ {\mathbb{R}^n}\to(0,\infty)$ be a variable
exponent function satisfying the globally
log-H\"{o}lder continuous condition and $A$
a general expansive matrix on ${\mathbb{R}^n}$. In this article,
the authors introduce the anisotropic variable Campanato-type
spaces and give some applications. Especially, using the
known atom and finite atom characterizations of anisotropic variable
Hardy space $H_A^{\vp}(\mathbb{R}^n)$, the authors prove that this Campanato-type space
is the appropriate dual space of $H_A^{\vp}(\mathbb{R}^n)$ with full range $p(\cdot)$.
As applications, the authors first deduce several equivalent
characterizations of these Campanato-type spaces. Furthermore,
the authors also introduce the anisotropic variable tent spaces and
show their atomic decomposition. Combining this and the
obtained dual theorem, the Carleson measure characterizations
of these anisotropic variable Campanato-type spaces are established.
}
\end{minipage}
\end{center}


\vspace{0.3cm}

\section{Introduction\label{s0}}

Recall that the classical Campanato space was first introduced
by Campanato \cite{c64} in 1964, which includes the bounded
mean oscillation function space $\mathop{\mathrm{BMO}}(\rn)$ of
John and Nirenberg \cite{jn61}. Later, Taibleson and Weiss
\cite{tw80} showed that the Campanato space is the dual space
of the well-known Hardy space $H^p(\rn)$ for $p\in(0,1]$ in 1980,
which generalized the celebrated dual
theorem of Fefferman and Stein \cite{fs72}. Namely,
the bounded mean oscillation function space
$\mathop{\mathrm{BMO}}(\rn)$ is the dual space of the
Hardy space $H^1(\rn)$. Nowadays, the Campanato space
plays an important role in harmonic analysis and
partial differential equations, and has been
systematically studied and developed so far.
For instance, Cianchi and Pick \cite{cp03} studied the Sobolev
embedding into Campanato spaces; El Baraka \cite{b06} established
the Littlewood--Paley function characterizations of
Campanato spaces;
Nakai \cite{n06} extended the Campanato spaces into the
spaces of homogeneous type; Nakai \cite{n10,n17} investigated
singular integral operators and fractional integral operators
on Campanato spaces or their predual spaces;
Nakai and Yoneda \cite{ny19} gave some applications of Campanato spaces
with variable growth condition to the Navier--Stokes equation;
Mizuta et al. \cite{nmos20,nmos20-1} recently considered Campanato--Morrey
spaces for the double phase functionals; Ho \cite{ho19} investigated
integral operators on BMO and Campanato spaces. For more developments
on Campanato spaces, we refer the
reader to \cite{mb03,hlyy,hyy,ns12,w}.

In addition, based on the variable Lebesgue space,
several variable function spaces have rapidly been developed
in the past two decades; see, for instance,
\cite{cw14,ho20,ho20-2,s13,s18,wx21,w21,x08,xy21}.
Note that, the study of variable Lebesgue spaces originates from
Orlicz \cite{o31}. However, they have been the subject of more
intensive study since the early 1990s due to their intrinsic interest
for applications into harmonic analysis, partial
differential equations and variational integrals with
nonstandard growth conditions (see books \cite{CUF} and \cite{DHHR}).
Recall that, let a measurable exponent function
$p(\cdot):\ \rn\to(0,\fz)$ satisfy the so-called
\emph{globally log-H\"{o}lder continuous condition}
[see \eqref{2e4} and \eqref{2e5} below], $p_- :={\essinf}_{x\in\rn}p(x)$,
and $p_+:={\esssup}_{x\in\rn}p(x)$. In 2012,
Nakai and Sawano \cite{ns12} introduced the variable Hardy
space $H^{p(\cdot)}(\rn)$ and generalized the Campanato spaces
into the variable exponents setting, in which they also showed
the variable Campanato space is the dual of $H^{p(\cdot)}(\rn)$
when $0< p_-\le p_+\le 1$; see \cite[Theorem 7.5]{ns12}.
After that, Sawano \cite{s13} improved the corresponding
result in [32] via extending the atomic characterization of $H^{p(\cdot)}(\rn)$.
Moreover, Zhuo et al. \cite{zyl} established equivalent
characterizations of $H^{p(\cdot)}(\rn)$ via intrinsic square functions. In particular,
Cruz-Uribe and Wang \cite{cw14} also independently investigated the variable Hardy space $H^{p(\cdot)}(\rn)$ with $p(\cdot)$
satisfying some weaker hypothesis than those used in \cite{ns12}.
Very recently, Cruz-Uribe et al. \cite{cmn20} further developed
a new approach to show norm inequalities for Calder\'{o}n--Zygmund singular
operators and fractional integral operators on variable Hardy spaces
with Muckenhoupt weight.

On another hand, under the framework of Hardy-type spaces associated
with ball quasi-Banach function spaces,
Zhang et al. \cite{zhyy} recently obtained the dual result
of $H^{p(\cdot)}(\rn)$ on the full range $0< p_-\le p_+<\fz$,
which extended the dual theorem of Nakai and Sawano \cite{ns12}.
Moreover, let $A$ be a general expansive matrix on
$\rn$. Recall that the variable Hardy space $\vh$ associated with $A$
was first introduced and studied by Liu et al. \cite{lwyy17},
in which they charactered $\vh$ in terms of maximal functions,
atoms, finite atoms, and Littlewood--Paley functions.
Furthermore, the anisotropic variable Campanato space
associated with $A$ was introduced by Wang \cite{w} very recently.
In \cite{w}, the author proved that the anisotropic
variable Campanato space is the dual space of the
anisotropic variable Hardy space $\vh$ with $0< p_-\le p_+\le 1$.
In addition, note that if $p_-\in(1,\fz)$, then $\vh=L^{p(\cdot)}(\rn)$
with equivalent quasi-norms (see \cite[Corollary 4.20]{zsy}). Thus,
in this case, it follows from Cruz-Uribe and Fiorenza \cite[Theorem 2.80]{CUF} that the dual space of
$\vh$ is just the variable Lebesgue space $L^{p(\cdot)'}(\rn)$,
where conjugate exponent function $p(\cdot)'$ is defined by setting
$\frac1{p(\cdot)}+\frac1{p(\cdot)'}=1.$ Obviously,
there is a gap between two ranges $0< p_-\le p_+\le 1$ and $p_-\in(1,\fz)$.
This means that when $0< p_-\le 1< p_+< \fz$, the dual space of $\vh$
is still missing until now. Indeed, the main difficult in this case is
that the considered Hardy space does not have a concave quasi-norm.

To solve this problem and also to enrich the theory of
anisotropic variable Campanato spaces,
in this article, by viewing the finite linear combinations of atoms as a whole,
we introduce the anisotropic variable Campanato-type spaces and
give some applications. Especially, we show that these
new introduced anisotropic variable Campanato-type spaces
include the anisotropic variable Campanato space in \cite{w}.
Moreover, inspired by \cite{zhyy} and using the known atom and finite atom characterizations
of $\vh$, we prove that this Campanato-type space
is the appropriate dual space of $\vh$ with $0< p_-\le p_+<\fz$,
which extends the known dual result of $\vh$. We point out that, even in the isotropic case,
the obtained dual theorem gives a
complete answer to the open question proposed by Izuki et al.
in \cite[Section 9.3]{ins13}.
As applications, we first deduce several equivalent
characterizations of these Campanato-type spaces. Furthermore,
we introduce the anisotropic variable tent spaces and
show their atomic decomposition. Combining this and the
obtained dual theorem, we finally establish
the Carleson measure characterizations of the
anisotropic variable Campanato-type spaces.

Precisely, this article is organized as follows.

In Section \ref{s1}, we recall some notions on dilations
and variable Lebesgue spaces. Then we introduce the
anisotropic variable Campanato-type space and show some
basic properties. Section \ref{s2} is devoted to proving
that the anisotropic variable Campanato-type space is
the dual space of $\vh$ for full range $0< p_-\le p_+<\fz$
(see Theorem \ref{2t1} below). To this end, we first
recall the atomic and the finite atomic characterizations
of the anisotropic variable Hardy space $\vh$ established
in \cite[Theorems 4.8 and 5.4]{lwyy17}. Combining these,
the special structure of the anisotropic variable
Campanato-type space, and some basic tools from
functional analysis, we identify the
Campanato-type space $\mathcal{L}_{\vp,q',s,\underline{p}}^A(\rn)$
with the dual space of the Hardy space $\vh$ for the full range
$0< p_-\le p_+<\fz$. We point out that, as a special case,
the dual theorem of $\vh$ with $0< p_-\le p_+\le 1$ is obviously
obtained, which cover the result
of \cite[Theorem 4.4]{w} [see Remark \ref{2r3} below].

In Section \ref{s3}, we first obtain
several equivalent characterizations of the anisotropic
variable Campanato-type space (see Theorem \ref{2t2}
and Corollary \ref{2c3} below). Then, by this and
the obtained dual result in Theorem \ref{2t1},
we further establish the Carleson measure
characterization of the anisotropic variable
Campanato-type space $\mathcal{L}_{\vp,1,s,\underline{p}}^A({\rn})$
(see Theorem \ref{2t3} below). To show this Carleson measure
characterization, we introduce the anisotropic variable tent
space (see Definition \ref{d1} below) and
give their atomic decomposition (see Lemma \ref{tent} below),
which plays key role in the proof of Theorem \ref{2t3}.
Indeed, applying this atomic decomposition, the Lusin area function
characterization of $\vh$ obtained in
\cite[Theorem 6.1]{lwyy17} (see also \cite[Theorem 4.4(i)]{lhy}),
and the obtained dual theorem, we further conclude Theorem \ref{2t3}.

Finally, we make some conventions on notation. Let $\nn:=\{1,2,\ldots\}$,
$\zz_+:=\{0\}\cup\nn$, and $\zz_+^n:=(\zz_+)^n$,
and use $\mathbf{0}$ to denote the \emph{origin} of $\rn$.
For any multi-index $\az:=(\az_1,\ldots,\az_n)\in\zz_+^n$
and $x:=(x_1,\ldots,x_n)\in\rn$,
let
$$|\az|:=\az_1+\cdots+\az_n,\ x^{\az}:=x_1^{\az_1}\cdots x_n^{\az_n},\ \mathrm{and}\
\partial^{\az}:=\lf(\frac{\partial}{\partial x_1}\r)^{\az_1} \cdots
\lf(\frac{\partial}{\partial x_n}\r)^{\az_n}.$$
We always denote by $C$
a \emph{positive constant} which is independent of the main parameters,
but it may vary from line to line. The notation $f\ls g$ means $f\le Cg$
and, if $f\ls g\ls f$, then we write $f\sim g$. If $f\le Cg$ and $g=h$
or $g\le h$, we then write $f\ls g\sim h$ or $f\ls g\ls h$, rather than
$f\ls g=h$ or $f\ls g\le h$. For any $q\in[1,\infty]$, we denote by $q'$ its
\emph{conjugate exponent}, namely, $1/q+1/q'=1$. The \emph{symbol} $\lfloor s\rfloor$ for
any $s\in\mathbb{R}$ denotes the largest integer not greater
than $s$. In addition, for any set $E\subset\rn$, we denote the
set $\rn\setminus E$ by $E^\complement$, its \emph{characteristic function}
by $\mathbf{1}_E$, and its \emph{n-dimensional Lebesgue measure} by $|E|$.
Throughout this article, the \emph{symbol} $C^{\fz}(\rn)$
denotes the set of all \emph{infinitely differentiable functions} on $\rn$.

\section{Anisotropic variable Campanato-type spaces\label{s1}}

In this section, we first recall some notions on dilations
and variable Lebesgue spaces. Then we introduce the
anisotropic variable Campanato-type space and give some
basic properties. To begin with, we present
the definition of dilations from \cite[p.\,5, Definition 2.1]{mb03}.

\begin{definition}\label{2d1}
A real $n\times n$ matrix $A$ is called an \emph{expansive matrix},
shortly, a \emph{dilation} if
$$\min_{\lz\in\sigma(A)}|\lz|>1,$$
here and thereafter, $\sigma(A)$ denotes the \emph{set of
all eigenvalues of $A$}.
\end{definition}

Let $b:=|\det A|$. Then, from \cite[p.\,6, (2.7)]{mb03}, it follows
that $b\in(1,\fz)$. By the fact that there exist an open
ellipsoid $\Delta$, with $|\Delta|=1$, and $r\in(1,\infty)$ such that
$\Delta\subset r\Delta\subset A\Delta$ (see \cite[p.\,5, Lemma 2.2]{mb03}),
we find that, for any $k\in\zz$, $B_k:=A^k\Delta$ is open,
$B_k\subset rB_k\subset B_{k+1}$ and $|B_k|=b^k$.
For any $x\in\rn$ and $k\in\mathbb{Z}$, an ellipsoid $x+B_k$
is called a \emph{dilated ball}. In what follows,
we always let $\mathfrak{B}$ be the set of all such
dilated balls, namely,
\begin{align}\label{2e1}
\mathfrak{B}:=\lf\{x+B_k:\ x\in\rn\ {\rm and}\ k\in\mathbb{Z}\r\}
\end{align}
and
\begin{align}\label{2e2}
\omega:=\inf\lf\{\ell\in\zz:\ r^\ell\ge2\r\}.
\end{align}

The following notion of homogeneous quasi-norms
is just \cite[p.\,6, Definition 2.3]{mb03}.

\begin{definition}\label{2d2}
A \emph{homogeneous quasi-norm},
associated with a dilation $A$, is a measurable mapping
$\rho:\ \rn \to [0,\infty)$ satisfying
\begin{enumerate}
\item[\rm{(i)}] if $x\neq\mathbf{0}$, then $\rho(x)\in(0,\fz)$;

\item[\rm{(ii)}] for any $x\in\rn$, $\rho(Ax)=b\rho(x)$;

\item[\rm{(iii)}] there exists an $H\in[1,\fz)$
such that, for any $x$, $y\in\rn$, $\rho(x+y)\le H[\rho(x)+\rho(y)]$.
\end{enumerate}
\end{definition}

In what follows, for a given dilation $A$, by \cite[p.\,6, Lemma 2.4]{mb03},
we may use, for both simplicity and convenience, the \emph{step homogeneous quasi-norm} $\rho$
defined by setting, for any $x\in\rn$,
\begin{equation*}\label{2e3'}
\rho(x):=\sum_{k\in\mathbb{Z}}
b^k{\mathbf 1}_{B_{k+1}\setminus B_k}(x)\hspace{0.25cm}
{\rm when}\ x\neq\mathbf{0},\hspace{0.35cm} {\rm or\ else}
\hspace{0.25cm}\rho(\mathbf{0}):=0.
\end{equation*}

A measurable function $p(\cdot):\ \rn\to(0,\fz]$ is called a
\emph{variable exponent}. Denote by $\cp(\rn)$ the \emph{set
of all variable exponents} $p(\cdot)$ satisfying
\begin{align}\label{2e3}
0<p_-:=\mathop\mathrm{ess\,inf}_{x\in \rn}p(x)\leq
\mathop\mathrm{ess\,sup}_{x\in \rn}p(x)=:p_+<\fz.
\end{align}
For any $p(\cdot)\in\cp(\rn)$ and a measurable function
$f$, the \emph{modular functional} $\varrho_{p(\cdot)}(f)$
is defined by setting
$$\varrho_{p(\cdot)}(f):=\int_\rn|f(x)|^{p(x)}\,dx$$
and the \emph{Luxemburg} (also called \emph{Luxemburg-Nakano})
\emph{quasi-norm} $\|f\|_{\lv}$ of $f$ is given by
\begin{equation*}
\|f\|_{\lv}:=\inf\lf\{\lz\in(0,\fz):\ \varrho_{p(\cdot)}(f/\lz)\le1\r\}.
\end{equation*}
Then the \emph{variable Lebesgue space} $\lv$ is defined to be the
set of all measurable functions $f$ such that $\varrho_{p(\cdot)}(f)<\fz$,
equipped with the quasi-norm $\|f\|_{\lv}$.

Let $C^{\log}(\rn)$ be the set of all functions $p(\cdot)\in\cp(\rn)$
satisfying the \emph{globally log-H\"older continuous condition}, namely,
there exist $C_{\log}(p)$, $C_\fz\in(0,\fz)$, and
$p_\fz\in\rr$ such that, for any $x,\ y\in\rn$,
\begin{equation}\label{2e4}
|p(x)-p(y)|\le \frac{C_{\log}(p)}{\log(e+1/\rho(x-y))}
\end{equation}
and
\begin{equation}\label{2e5}
|p(x)-p_\fz|\le \frac{C_\fz}{\log(e+\rho(x))}.
\end{equation}

In what follows, for any $s\in\zz_+$, $\mathbb{P}_s(\rn)$
denotes the set of all polynomials on $\rn$
with degree not greater than $s$; for any ball $B\in\BB$
with $\BB$ as in \eqref{2e1}
and any locally integrable function $g$ on $\rn$,
we use $P^s_Bg$ to denote the \emph{minimizing polynomial} of
$g$ with degree not greater than $s$, which means
that $P^s_Bg$ is the unique polynomial $f\in\mathbb{P}_s(\rn)$
such that, for any $h\in\mathbb{P}_s(\rn)$,
$$\int_{B}[g(x)-f(x)]h(x)\,dx=0.$$

We next introduce the following anisotropic variable
Campanato-type spaces. For any given $q\in[1,{\infty})$,
we use $L_{\loc}^q(\rn)$ to denote the set of all $q$-order
locally integrable functions on $\rn$.

\begin{definition}\label{2d5}
Let $p(\cdot)\in \mathcal{P}(\rn)$, $q\in[1,{\infty})$, $\eta\in(0,\infty)$,
and $s\in\zz_+$. Then the \emph{anisotropic variable Campanato-type space}
$\mathcal{L}_{\vp,q,s,\eta}^A(\rn)$ is defined to be
the set of all $f\in L^q_{\rm loc}(\rn)$ such that
\begin{align*}
\|f\|_{\mathcal{L}_{\vp,q,s,\eta}^A(\rn)}
&:=\,\sup\lf\|\lf\{\sum_{i=1}^m
\lf[\frac{{\lambda}_i}{\|{\mathbf{1}}_{B^{(i)}}\|_{\lv}}\r]^{\eta}
{\mathbf{1}}_{B^{(i)}}\r\}^{\frac1{\eta}}\r\|_{\lv}^{-1}
\\
&\quad\times\sum_{j=1}^m\lf\{\frac{{\lambda}_j|B^{(j)}|}{\|{\mathbf{1}}_{B^{(j)}}
\|_{\lv}}
\lf[
\frac1{|B^{(j)}|}\int_{B^{(j)}}\lf|f(x)-P^s_{B^{(j)}}f(x)\r|^q \,dx\r]^\frac1q\r\}
\end{align*}
is finite, where the supremum is taken over all $m\in\nn$, $\{B^{(j)}\}_{j=1}^m\subset\BB$, and
$\{\lambda_j\}_{j=1}^m\subset[0,\infty)$ with $\sum_{j=1}^m\lambda_j\neq0$.
\end{definition}

\begin{remark}\label{2r1}
\begin{enumerate}
\item[{\rm(i)}]
We point out that $\mathbb{P}_s(\rn)\subset\mathcal{L}_{\vp,q,s,\eta}^A(\rn)$.
Indeed, $$\|f\|_{\mathcal{L}_{\vp,q,s,\eta}^A(\rn)}=0$$ if and only if
$f\in\mathbb{P}_s(\rn)$. Throughout this
article, we always identify $f\in \mathcal{L}_{\vp,q,s,\eta}^A(\rn)$
with $\{f+P:\ P\in\mathbb{P}_s(\rn)\}$.
\item[{\rm(ii)}]
For any $f\in L^q_{\rm loc}(\rn)$, define
\begin{align*}
\||f\||_{\mathcal{L}_{\vp,q,s,\eta}^A(\rn)}
&:=\,\sup\inf
\lf\|\lf\{\sum_{i=1}^m
\lf[\frac{{\lambda}_i}{\|{\mathbf{1}}_{B^{(i)}}\|_{\lv}}\r]^{\eta}
{\mathbf{1}}_{B^{(i)}}\r\}^{\frac1{\eta}}\r\|_{\lv}^{-1}
\\
&\quad\times\sum_{j=1}^m\lf\{
\frac{{\lambda}_j|B^{(j)}|}{\|{\mathbf{1}}_{B^{(j)}}
\|_{\lv}}\lf[
\frac1{|B^{(j)}|}\int_{B^{(j)}}\lf|f(x)-P(x)\r|^q \,dx\r]^\frac1q\r\},
\end{align*}
where the supremum is the same as in Definition \ref{2d5} and
the infimum is taken over all $P\in\mathbb{P}_s(\rn)$.
Then we can easily show that $\||\cdot\||_{\mathcal{L}_{\vp,q,s,\eta}^A(\rn)}$ is an equivalent quasi-norm
of the Campanato-type space ${\mathcal{L}_{\vp,q,s,\eta}^A(\rn)}$
and we omit the details here.
\end{enumerate}
\end{remark}

For the Campanato-type space $\mathcal{L}_{\vp,q,s,\eta}^A(\rn)$,
we have the following equivalent quasi-norm characterization.

\begin{proposition}\label{2p1}
Let $p(\cdot)\in \mathcal{P}(\rn)$, $q\in[1,{\infty})$, $\eta\in(0,\infty)$,
and $s\in\zz_+$. For any $f\in L^q_{\rm loc}(\rn)$, define
\begin{align*}
\widetilde{\|f\|}_{\mathcal{L}_{\vp,q,s,\eta}^A(\rn)}
&:=\,\sup
\lf\|\lf\{\sum_{i\in\nn}
\lf[\frac{{\lambda}_i}{\|{\mathbf{1}}_{B^{(i)}}\|_{\lv}}\r]^{\eta}
{\mathbf{1}}_{B^{(i)}}\r\}^{\frac1{\eta}}\r\|_{\lv}^{-1}\\
&\quad\times\sum_{j\in\nn}\lf\{\frac{{\lambda}_j|B^{(j)}|}{\|{\mathbf{1}}_{B^{(j)}}
\|_{\lv}}
\lf[
\frac1{|B^{(j)}|}\int_{B^{(j)}}\lf|f(x)-P^s_{B^{(j)}}f(x)\r|^q \,dx\r]^\frac1q\r\},
\end{align*}
where the supremum is taken over all $\{B^{(j)}\}_{j\in\nn}\subset \BB$ and
$\{\lambda_j\}_{j\in\nn}\subset[0,\infty)$ satisfying
\begin{align}\label{2e6}
\lf\|\lf\{\sum_{j\in\nn}
\lf[\frac{{\lambda}_j}{\|{\mathbf{1}}_{B^{(j)}}\|_\lv}\r]^{\eta}
{\mathbf{1}}_{B^{(j)}}\r\}^{\frac1{\eta}}\r\|_{\lv}\in(0,\infty).
\end{align}
Then, for any $f\in L^q_{\rm loc}(\rn)$,
$$\widetilde{\|f\|}_{\mathcal{L}_{\vp,q,s,\eta}^A(\rn)}
=\|f\|_{\mathcal{L}_{\vp,q,s,\eta}^A(\rn)}.$$
\end{proposition}

\begin{proof}
Let $\vp$, $q$, $s$, and $\eta$ be as in the present proposition and
$f\in L^q_{\rm loc}({{\rr}^n})$. Obviously,
$$\|f\|_{\mathcal{L}_{\vp,q,s,\eta}(\rn)}\le \widetilde{\|f\|}_{\mathcal{L}_{\vp,q,s,\eta}(\rn)}.$$
Conversely, let $\{B^{(j)}\}_{j\in\nn}\subset \BB$ and
$\{\lambda_j\}_{j\in\nn}\subset[0,\infty)$ satisfy \eqref{2e6}.
Observe that
\begin{align*}
&\lim_{m\to\fz}
\lf\|\lf\{\sum_{i=1}^m
\lf[\frac{{\lambda}_i}{\|{\mathbf{1}}_{B^{(i)}}\|_{\lv}}\r]^{\eta}
{\mathbf{1}}_{B^{(i)}}\r\}^{\frac1{\eta}}\r\|_{\lv}^{-1}\\
&\hs\hs\hs\times\sum_{j=1}^m\lf\{\frac{{\lambda}_j|B^{(j)}|}{\|{\mathbf{1}}_{B^{(j)}}
\|_{\lv}}\lf[
\frac1{|B^{(j)}|}\int_{B^{(j)}}\lf|f(x)-P^s_{B^{(j)}}f(x)\r|^q \,dx\r]^\frac1q\r\}\\
&\hs=\lf\|\lf\{\sum_{i\in\nn}
\lf[\frac{{\lambda}_i}{\|{\mathbf{1}}_{B^{(i)}}\|_{\lv}}\r]^{\eta}
{\mathbf{1}}_{B^{(i)}}\r\}^{\frac1{\eta}}\r\|_{\lv}^{-1}\\
&\hs\hs\hs\times\sum_{j\in\nn}\lf\{\frac{{\lambda}_j|B^{(j)}|}{\|{\mathbf{1}}_{B^{(j)}}\|_{\lv}}
\lf[\frac1{|B^{(j)}|}\int_{B^{(j)}}\lf|f(x)-P^s_{B^{(j)}}f(x)\r|^q \,dx\r]^\frac1q\r\}.
\end{align*}
Therefore, for any given $\varepsilon\in(0,\fz)$,
there exists an $m_0\in\nn$ such that
$\sum_{j=1}^{m_0}\lambda_j\neq0$
and
\begin{align*}
&\lf\|\lf\{\sum_{i\in\nn}
\lf[\frac{{\lambda}_i}{\|{\mathbf{1}}_{B^{(i)}}\|_{\lv}}\r]^{\eta}
{\mathbf{1}}_{B^{(i)}}\r\}^{\frac1{\eta}}\r\|_{\lv}^{-1}\\
&\hs\hs\hs\times\sum_{j\in\nn}\lf\{\frac{{\lambda}_j|B^{(j)}|}{\|{\mathbf{1}}_{B^{(j)}}
\|_{\lv}}\lf[
\frac1{|B^{(j)}|}\int_{B^{(j)}}\lf|f(x)-P^s_{B^{(j)}}f(x)\r|^q \,dx\r]^\frac1q\r\}\\
&\hs <\lf\|\lf\{\sum_{i=1}^{m_0}
\lf[\frac{{\lambda}_i}{\|{\mathbf{1}}_{B^{(i)}}\|_{\lv}}\r]^{\eta}
{\mathbf{1}}_{B^{(i)}}\r\}^{\frac1{\eta}}\r\|_{\lv}^{-1}\\
&\hs\hs\hs\times\sum_{i=1}^{m_0}\lf\{\frac{{\lambda}_j|B^{(j)}|}{\|{\mathbf{1}}_{B^{(j)}}
\|_{\lv}}\lf[
\frac1{|B^{(j)}|}\int_{B^{(j)}}\lf|f(x)-P^s_{B^{(j)}}f(x)\r|^q \,dx\r]^\frac1q\r\}
+\varepsilon\\
&\hs\le\|f\|_{\mathcal{L}_{\vp,q,s,\eta}^A(\rn)}+\varepsilon,
\end{align*}
which, combined with the arbitrariness of $\{B^{(j)}\}_{j\in\nn}\subset \BB$ and
$\{\lambda_j\}_{j\in\nn}\subset[0,\infty)$ satisfying \eqref{2e6}, and $\varepsilon\in(0,\fz)$,
further implies that
$$\widetilde{\|f\|}_{\mathcal{L}_{\vp,q,s,\eta}^A(\rn)}
\le\|f\|_{\mathcal{L}_{\vp,q,s,\eta}^A(\rn)}.$$
This finishes the proof of Proposition \ref{2p1}.
\end{proof}

Recall that the following
anisotropic variable Campanato space $\mathcal{L}_{\vp,q,s}^A(\rn)$
was introduced in \cite[Definition 4.1]{w}.

\begin{definition}\label{2d6}
Let $p(\cdot)\in \mathcal{P}(\rn)$, $q\in[1,\fz)$, and $s\in\zz_+$.
The \emph{anisotropic variable Campanato space}
$\mathcal{L}_{\vp,q,s}^A(\rn)$ is defined to be
the set of all $f\in L^q_{\rm loc}({{\rr}^n})$ such that
$$\|f\|_{\mathcal{L}_{\vp,q,s}^A(\rn)}
:=\sup_{B\in\mathfrak{B}}\inf_{P\in\mathbb{P}_s(\rn)}
\frac{|B|}{\|{\mathbf 1}_B\|_{\lv}}\lf[\frac1{|B|}
\int_B\lf|f(x)-P(x)\r|^q\,dx\r]^{\frac1q}<\fz.$$
\end{definition}

\begin{remark}\label{2r2}
\begin{enumerate}
\item[{\rm(i)}]
For any $f\in L^q_{\rm loc}({{\rr}^n})$, define
$$\||f\||_{\mathcal{L}_{\vp,q,s}^A(\rn)}
:=\sup_{B\in \mathfrak{B}}\frac{|B|}{\|{\mathbf{1}}_{B}\|_{\lv}}\lf[
\frac1{|B|}\int_{B}\lf|f(x)-P_B^sf(x)\r|^q \,dx\r]^\frac1q.$$
It is easy to see that $\||\cdot\||_{\mathcal{L}_{\vp,q,s}^A(\rn)}$
is an equivalent quasi-norm of the Campanato space
$\mathcal{L}_{\vp,q,s}^A(\rn)$.

\item[{\rm(ii)}]
From Remark \ref{2r1}(ii) and Definition \ref{2d6},
it immediately follows that
$$\mathcal{L}_{\vp,q,s,\eta}^A(\rn)\subset\mathcal{L}_{\vp,q,s}^A(\rn)$$
for all indices as in Definition \ref{2d5},
and this inclusion is continuous.
\end{enumerate}
\end{remark}

We now investigate the further relation between
two spaces ${\mathcal{L}_{\vp,q,s,\eta}^A(\rn)}$ and
$\mathcal{L}_{\vp,q,s}^A(\rn)$.
Indeed, the next proposition shows that the inverse inclusion of
Remark \ref{2r2}(ii) also holds true for certain ranges of indices.

\begin{proposition}\label{2p2}
Let $p(\cdot)\in \mathcal{P}(\rn)$, $p_+\in(0,1]$ with $p_+$ as in
\eqref{2e3}, $q\in[1,\fz)$, $\eta\in(0,1]$, and $s\in\zz_+$.
Then $$\mathcal{L}_{\vp,q,s,\eta}^A(\rn)=\mathcal{L}_{\vp,q,s}^A(\rn)$$
with equivalent quasi-norms.
\end{proposition}

\begin{proof}
Let $p(\cdot)$, $p_+$, $q$, $\eta$, and $s$ be as in the present proposition.
To prove this proposition, by Remark \ref{2r2}(ii), it suffices to show that
$\mathcal{L}_{\vp,q,s}^A(\rn)\subset\mathcal{L}_{\vp,q,s,\eta}^A(\rn)$
and the inclusion is continuous. For this purpose, let $f\in
\mathcal{L}_{\vp,q,s}^A(\rn)$. Then, from the assumptions that
$p_+\in(0,1]$ and $\eta\in(0,1]$, together with the monotonicity of
$\ell^{\eta}$ and Remark \ref{2r2}(i), we deduce that
\begin{align*}
\|f\|_{\mathcal{L}_{\vp,q,s,\eta}^A(\rn)}
&\ls \sup
\lf(\sum_{i=1}^m{\lambda}_i\r)^{-1}
\sum_{j=1}^m\lf\{
\frac{{\lambda}_j|B^{(j)}|}
{\|{\mathbf{1}}_{B^{(j)}}\|_{\lv}}\lf[\frac1{|B^{(j)}|}\int_{B^{(j)}}\lf|
f(x)-P^s_{B^{(j)}}f(x)\r|^q \,dx\r]^\frac1q\r\}\\
&\ls \sup\lf(\sum_{i=1}^m\lambda_i\r)^{-1}
\sum_{j=1}^m{\lambda}_j
\||f\||_{\mathcal{L}_{\vp,q,s}^A(\rn)}
\sim \|f\|_{\mathcal{L}_{\vp,q,s}^A(\rn)},
\end{align*}
where the supremum is taken over all $m\in\nn$, $\{B^{(j)}\}_{j=1}^m\subset \BB$, and
$\{\lambda_j\}_{j=1}^m\subset[0,\infty)$ with $\sum_{j=1}^m\lambda_j\neq0$.
This finishes the proof of Proposition \ref{2p2}.
\end{proof}

\begin{remark}
By Proposition \ref{2p2} with $p(\cdot)\equiv p\in(0,1]$ and \cite[Remark 11(iii)]{hlyy19},
we find that the new introduced Campanato-type space as
in Definition \ref{2d5} includes the classical
Campanato space $L_{\frac1p-1,\,q,\,s}(\rn)$,
introduced by Campanato \cite{c64}, and the
space $\mathop{\mathrm{BMO}}(\rn)$, introduced by
John and Nirenberg \cite{jn61}, as special cases.
\end{remark}

\section{Duality between $\vh$ and $\mathcal{L}_{\vp,q',s,\underline{p}}^A(\rn)$\label{s2}}

In this section, we establish the relation between the
anisotropic variable Campanato-type space and the
anisotropic variable Hardy space
via duality. To be precise, based on the structure of the anisotropic
variable Campanato-type space introduced in Definition
\ref{2d5}, and applying both the atomic and
the finite atomic characterizations of the anisotropic variable
Hardy space $\vh$ established in \cite{lwyy17} as well as some basic
tools from functional analysis, we show that the new
introduced Campanato-type space
$\mathcal{L}_{\vp,q',s,\underline{p}}^A(\rn)$ is the dual space of
the Hardy space $\vh$ for the full range $p(\cdot)\in C^{\log}(\rn)$,
which extends the known dual result of \cite{w}.

Recall that a $C^\infty(\rn)$
function $\varphi$ is called a \emph{Schwartz function} if,
for any $m\in\zz_+$ and multi-index $\az\in\zz_+^n$,
$$\|\varphi\|_{\alpha,m}:=
\sup_{x\in\rn}[\rho(x)]^m
|\partial^\alpha\varphi(x)|<\infty.$$
Denote by
$\cs(\rn)$ the set of all Schwartz functions, equipped
with the topology determined by
$\{\|\cdot\|_{\alpha,m}\}_{\az\in\zz_+^n,m\in\zz_+}$,
and $\cs'(\rn)$ the \emph{dual space} of $\cs(\rn)$, equipped
with the weak-$\ast$ topology.
For any $N\in\mathbb{Z}_+$, let
$$\cs_N(\rn):=\{\varphi\in\cs(\rn):\
\|\varphi\|_{\alpha,\ell}\leq1,\
|\alpha|\leq N,\ \ell\leq N\},$$
equivalently,
\begin{align*}
\varphi\in\cs_N(\rn)\Longleftrightarrow
\|\varphi\|_{\cs_N(\rn)}:=\sup_{|\alpha|\leq N}
\sup_{x\in\rn}\lf[\lf|\partial^\alpha
\varphi(x)\r|\max\lf\{1,\lf[
\rho(x)\r]^N\r\}\r]\leq1.
\end{align*}
Let $\lambda_-$, $\lambda_+\in(1,\fz)$
be two \emph{numbers} such that
$$\lambda_-<\min\{|\lambda|:\
\lambda\in\sigma(A)\}
\leq\max\{|\lambda|:\
\lambda\in\sigma(A)\}<\lambda_+.$$
We should point out that, if $A$ is diagonalizable over
$\mathbb{C}$, then we may let
$\lambda_-:=\min\{|\lambda|:\
\lambda\in\sigma(A)\}$
and
$\lambda_+:=\max\{|\lambda|:\
\lambda\in\sigma(A)\}$.
Otherwise, we may choose them sufficiently close to these equalities
in accordance with what we need in our arguments.

Applying \cite[Definition 2.4 and Theorem 3.10]{lwyy17},
we now recall the following equivalent definition
of the anisotropic variable Hardy space $\vh$
via the radial maximal function.

\begin{definition}\label{2d4}
Let $p(\cdot)\in C^{\log}(\rn)$ and $\varphi\in\cs(\rn)$ satisfy $\int_{\rn}\varphi(x)\,dx\neq0$.
The \emph{anisotropic variable Hardy space} $\vh$ is defined
to be the set of all $f\in\cs'(\rn)$ such that
$$\|f\|_{\vh}:=\lf\| M_\varphi^0(f)\r\|_{\lv}<\fz,$$
where $M_\varphi^0(f)$ is defined by setting, for any $x\in\rn$,
\begin{equation*}
M_\varphi^0(f)(x):= \sup_{k\in\zz}
\lf|f\ast\varphi_k(x)\r|,
\end{equation*}
here and thereafter, for any $k\in\zz$,
$\varphi_k(\cdot):=b^{k}\varphi(A^{k}\cdot)$.
\end{definition}

We also need the following notions
of anisotropic variable $(\vp,r,s)$-atoms and anisotropic
variable finite atomic Hardy spaces from \cite[Definitions 4.1 and 5.1]{lwyy17}.

\begin{definition}\label{3d1}
Let $\vp\in\mathcal{P}(\rn)$, $r\in(1,\fz]$, and $s$
be as in \eqref{4e1}.
A measurable function $a$ on $\rn$ is called an
\emph{anisotropic variable $(\vp,r,s)$-atom} if
\begin{enumerate}
\item[{\rm (i)}] $\supp a:=\{x\in\rn:\ a(x)\neq0\} \subset B\in\mathfrak{B}$;
\item[{\rm (ii)}] $\|a\|_{L^r(\rn)}\le \frac{|B|^{1/r}}{\|{\mathbf 1}_B\|_{\lv}}$;
\item[{\rm (iii)}] for any $\gamma\in\zz_+^n$ with $|\gamma|\le s$,
$\int_{\rn}a(x)x^\gamma\,dx=0$.
\end{enumerate}
\end{definition}

\begin{definition}\label{3d2}
Let $p(\cdot)\in C^{\log}(\rn)$, $r\in(1,\fz]$, and $s$ be as in \eqref{4e1}.
The \emph{anisotropic variable finite atomic Hardy space}
$\vfah$ is defined to be the set of all
$f\in\cs'(\rn)$ satisfying that there exist
an $I\in\nn$,
$\{\lz_i\}_{i=1}^I\subset[0,\fz)$, and
a finite sequence $\{a_i\}_{i=1}^I$ of $(\vp,r,s)$-atoms
supported, respectively, in
$\{B^{(i)}\}_{i=1}^I\subset\mathfrak{B}$
such that $f=\sum_{i=1}^I\lambda_ia_i$ in $\cs'(\rn)$.
Moreover, for any $f\in\vfah$, let
\begin{align*}
\|f\|_{\vfah}:=
{\inf}\lf\|\lf\{\sum_{i=1}^{I}
\lf[\frac{\lz_i}{\|{\mathbf 1}_{B^{(i)}}\|_{\lv}}\r]^
{\underline{p}}{\mathbf 1}_{B^{(i)}}\r\}^{1/\underline{p}}\r\|_{\lv},
\end{align*}
where the infimum is taken over all decompositions of $f$ as above and,
here and thereafter,
$\underline{p}:=\min\{p_-,1\}$ with $p_-$ as in \eqref{2e3}.
\end{definition}

Next, we state the dual result between $\vh$ and
$\mathcal{L}_{\vp,q',s,\underline{p}}^A(\rn)$ as follows.

\begin{theorem}\label{2t1}
Let $p(\cdot)\in C^{\log}(\rn)$, $q\in(\max\{1,p_+\},{\infty}]$ with
$p_+$ as in \eqref{2e3}, and \begin{align}\label{4e1}
s\in\lf[\lf\lfloor\lf(\dfrac1{p_-}-1\r)
\dfrac{\ln b}{\ln\lambda_-}\r\rfloor,\fz\r)\cap\zz_+,
\end{align}
where $\lambda_-$ and $p_-$ are as in \eqref{2e5} and \eqref{2e3},
respectively. Then the dual space of $\vh$, denoted by $(\vh)^*$,
is $\mathcal{L}_{\vp,q',s,\underline{p}}^A(\rn)$ in the following sense:
\begin{enumerate}
\item[{\rm (i)}] Let $g\in\mathcal{L}_{\vp,q',s,\underline{p}}^A(\rn)$.
Then the linear functional
\begin{align}\label{2te1}
\Lambda_g:\ f\rightarrow \Lambda_g(f):=\int_{\rn}f(x)g(x)\,dx,
\end{align}
initially defined for any $f\in H_{A,{\rm fin}}^{\vp,q,s}(\rn)$,
has a bounded extension to $\vh$.

\item[{\rm (ii)}] Conversely, any continuous linear
functional on $\vh$ arises as in \eqref{2te1}
with a unique $g\in\mathcal{L}_{\vp,q',s,\underline{p}}^A(\rn)$.
\end{enumerate}
Moreover,
$\|g\|_{\mathcal{L}_{\vp,q',s,\underline{p}}^A(\rn)}\sim\|\Lambda_g\|_{(\vh)^*}$,
where the positive equivalence constants
are independent of $g$.
\end{theorem}

To prove Theorem \ref{2t1}, we need the atomic and the
finite atomic characterizations of $\vh$ as follows,
which were established in \cite[Theorems 4.8 and 5.4]{lwyy17}.

\begin{lemma}\label{3l1}
Let $\vp$, $q$, and $s$ be as in Theorem \ref{2t1},
$\{a_j\}_{j\in\nn}$ a sequence of $(\vp,q,s)$-atoms supported,
respectively, in $\{B^{(j)}\}_{j\in\nn}\subset\BB$, and
$\{{\lambda}_j\}_{j\in\nn}\subset[0,\infty)$ such that
$$\left\|\left\{\sum_{j\in\nn}
\left[\frac{{\lambda}_j}{\|{\mathbf{1}}_{B^{(j)}}
\|_{\lv}}\right]^{\underline{p}}\mathbf{1}_{B^{(j)}}
\right\}^{\frac1{\underline{p}}}\right\|_{\lv}<\fz.$$
Then the series $f=\sum_{j\in\nn}{\lambda}_ja_j$
converges in $\vh$, $f\in \vh$, and there
exists a positive constant $C$, independent of $f$, such that
$$\lf\|f\r\|_{\vh}\le C\left\|\left\{\sum_{j\in\nn}
\left[\frac{{\lambda}_j}{\|{\mathbf{1}}_{B^{(j)}}
\|_{\lv}}\right]^{\underline{p}}\mathbf{1}_{B^{(j)}}
\right\}^{\frac1{\underline{p}}}\right\|_{\lv}.$$
\end{lemma}

\begin{lemma}\label{3l2}
Let $p(\cdot)\in C^{\log}(\rn)$ and $s$ be as in \eqref{4e1}.
\begin{enumerate}
\item[{\rm (i)}]
If $r\in(\max\{1,p_+\},\fz)$ with $p_+$ as in
\eqref{2e3}, then $\|\cdot\|_{\vfah}$
and $\|\cdot\|_{\vh}$ are equivalent quasi-norms on $\vfah$;
\item[{\rm (ii)}]
$\|\cdot\|_{\vfahfz}$
and $\|\cdot\|_{\vh}$ are equivalent quasi-norms on
$\vfahfz\cap C(\rn)$,
where $C(\rn)$ denotes the set of all continuous functions
on $\rn$.
\end{enumerate}
\end{lemma}

The following lemma is also needed for establishing
the dual theorem, whose proof is similar to
that of \cite[Proposition 3.13]{zhyy}; we omit the details.

\begin{lemma}\label{3l3}
Let $\vp$ and $s$ be as in Lemma \ref{3l2}. Then
$\vfahfz\cap C(\rn)$ is dense in $\vh$.
\end{lemma}

We now show Theorem \ref{2t1}.

\begin{proof}[Proof of Theorem \ref{2t1}]
Let all the notation be as in the present theorem.
We first prove (i) by considering the range of $q$
in the following two cases.

\emph{\textbf{Case 1}}. $q\in(\max\{1,p_+\},\infty)$. In this case, let $g\in\mathcal{L}_{\vp,q',s,\underline{p}}^A(\rn)$.
For any $f\in H_{A,\rm fin}^{\vp,q,s}(\rn)$,
from Definition \ref{3d2}, it follows that there exist
$\{\lambda_j\}_{j=1}^m\subset[0,\infty)$ and
$\{a_j\}_{j=1}^m$ of $(\vp,q,s)$-atoms supported, respectively, in the balls
$\{B^{(j)}\}_{j=1}^m\subset\BB$ such that $f=\sum_{j=1}^m\lambda_ja_j$ in $\cs'(\rn)$ and
$$
\lf\|\lf\{\sum_{j=1}^{m} \lf[\frac{\lambda_j}
{\|\mathbf{1}_{B^{(j)}}\|_{\lv}} \r]^{\underline{p}}\mathbf{1}_{B^{(j)}} \r\}
^{\frac{1}{\underline{p}}}\r\|_{\lv}
\sim\|f\|_{H_{A,\rm fin}^{\vp,q,s}(\rn)}.
$$
By this, the vanishing moments of $a_j$,
the H\"{o}lder inequality, the size condition of $a_j$,
Remark \ref{2r1}(ii), and Lemma \ref{3l2}(i), we conclude that
\begin{align*}
|\Lambda_g(f)|&
\le\sum_{j=1}^{m}\lambda_j\lf|\int_{\rn}a_j(x)g(x)\,dx\r|\\
&=\sum_{j=1}^{m}\lambda_j\inf_{P\in \cp_s(\rn)}\lf|\int_{B^{(j)}}a_j(x)\lf[g(x)-P(x)\r]\,dx\r|\noz\\
&\le\sum_{j=1}^{m}\frac{\lambda_j|B^{(j)}|}{\|\mathbf{1}_{B^{(j)}}\|_{\lv}}
\inf_{P\in \cp_s(\rn)}\lf[\frac1{|B^{(j)}|}
\int_{B^{(j)}}|g(x)-P(x)|^{q'}\,dx\r]^{\frac{1}{q'}}\noz\\
&\ls\|g\|_{\mathcal{L}_{\vp,q',s,\underline{p}}^A(\rn)}
\lf\|\lf\{\sum_{i=1}^{m}
\lf[\frac{{\lambda}_i}{\|{\mathbf{1}}_{B^{(i)}}\|_{\lv}}\r]^{\underline{p}}
{\mathbf{1}}_{B^{(i)}}\r\}^{\frac1{\underline{p}}}\r\|_{\lv}\noz\\
&\sim\|g\|_{\mathcal{L}_{\vp,q',s,\underline{p}}^A(\rn)}
\|f\|_{H_{A,\rm fin}^{\vp,q,s}(\rn)}
\sim\|g\|_{\mathcal{L}_{\vp,q',s,\underline{p}}^A(\rn)}
\|f\|_{\vh},\noz
\end{align*}
which, together with the fact that $H_{A,\rm fin}^{\vp,q,s}(\rn)$
is dense in $\vh$, further implies that (i) holds true in this case.

\emph{\textbf{Case 2}}. $q=\fz$. In this case, using Lemma \ref{3l3}
and repeating the proof of \emph{Case 1}, we then find that any
$g\in\mathcal{L}_{\vp,1,s,\underline{p}}^A(\rn)$ induces a
bounded linear functional on $\vh$, which is initially defined on
$H_{A,\rm fin}^{\vp,\fz,s}(\rn)\cap C{(\rn)}$
by setting, for any
$f\in H_{A,\rm fin}^{\vp,\fz,s}(\rn)\cap C{(\rn)}$,
\begin{equation}\label{opl}
\Lambda_g:\ f\mapsto\ \Lambda_g(f):=\int_{\rn}f(x)g(x)\,dx,
\end{equation}
and then has a bounded extension to $\vh$. Thus, to show (i) in this case,
it remains to prove that, for any $f\in H_{A,\rm fin}^{\vp,\fz,s}(\rn)$,
\begin{equation}\label{oo9}
\Lambda_g(f)=\int_{\rn}f(x)g(x)\,dx.
\end{equation}
For this purpose, assume that $f\in H_{A,\rm fin}^{\vp,\fz,s}(\rn)$
and $\supp(f)\subset B(\mathbf{0},L)$ with some $L\in(0,\infty)$.
Let $\varphi\in\mathcal{S}(\rn)$ satisfy
$\supp\varphi\subset B(\mathbf{0},1)$
and $\int_{\rn}\varphi(x)\,dx=1$. Thus, for any $t\in(0,1)$,
$\varphi_t\ast f\in H_{A,\rm fin}^{\vp,\fz,s}(\rn)\cap C{(\rn)}$,
here and thereafter, for any $t\in(0,\fz)$, $\varphi_{t}(\cdot):=t^{-n}\varphi(t^{-1}\cdot)$.
Letting $r\in(\max\{1,p_+\},\infty)$, then
$f\in L^r(\rn)$ and hence, from
\cite[Theorem 2.1]{d01}, it follows that
\begin{equation*}\label{oo10}
\lim_{t\in(0,\infty),t\to0}
\lf\|f-\varphi_t\ast f\r\|_{L^r(\rn)}=0.
\end{equation*}
By this and the Riesz lemma, we know that
there exists a sequence
$\{t_k\}_{k\in\nn}\subset(0,1)$ such that $\lim_{k\to\infty}t_k=0$
and, for almost every $x\in\rn$,
$\lim_{k\to\infty}\varphi_{t_k}\ast f(x)=f(x)$.
Therefore, we have
\begin{align}\label{3e1}
\lim_{k\to\infty}\|f-\varphi_{t_k}\ast f\|_{\vh}=0.
\end{align}
Indeed, we only need to show that, for any given
$(p(\cdot),\infty,s)$-atom, \eqref{3e1} holds true.
To this end, let $a$ be a $(p(\cdot),\infty,s)$-atom
supported in the ball $B(x_B,\ell(B))$. Then, applying
\cite[Theorem 2.1]{d01}, we have
\begin{equation}\label{fin5r}
\lim_{t\in(0,1),t\to0}\lf\|a-\varphi_{t}\ast a\r\|_{L^r(\rn)}=0,
\end{equation}
Note that, for any $t\in(0,1)$,
$\frac{|B(x_B,\ell(B)+2)|^{\frac{1}{r}}(a-\varphi_{t}\ast a)}{\|\mathbf1_{B(x_B,\ell(B)+2)}\|_{L^{p(\cdot)}(\rn)}\|a-\varphi_t\ast a\|_{L^r(\rn)}}$ is an $(p(\cdot),r,s)$-atom supported in the ball $B(x_B,\ell(B)+2)$. Combining this, Lemma \ref{3l1}, and \eqref{fin5r},
we further conclude that
\begin{align*}\label{zzv}
\lf\|a-\varphi_t\ast a\r\|_{\vh}
&\lesssim\frac{\|\mathbf1_{B(x_B,\ell(B)+2)}\|_{\lv}\|a-\varphi_t\ast a\|_{L^r(\rn)}}{|B(x_B,\ell(B)+2)|^{\frac{1}{r}}}\\
&\lesssim\|a-\varphi_t\ast a\|_{L^r(\rn)}\to0\noz
\end{align*}
as $t\to 0$. This implies that \eqref{3e1} holds true.

Furthermore, from \eqref{3e1}, \eqref{opl}, the fact that
$$\lf|\lf(\varphi_{t_k}\ast f\r)g\r|\le\|f\|_{L^\infty(\rn)}\mathbf1_{B(\mathbf{0},L+1)}|g|\in L^1(\rn),$$
and the Lebesgue dominated convergence theorem,
we deduce that
\begin{equation*}
\Lambda_g(f)=\lim_{k\to\infty}\Lambda_g(\varphi_{t_k}\ast f)
=\lim_{k\to\infty}\int_{\rn}\varphi_{t_k}\ast f(x)g(x)\,dx=\int_{\rn}f(x)g(x)\,dx,
\end{equation*}
which completes the proof of \eqref{oo9} and hence of (i).

We now show (ii). For any $\Lambda\in(\vh)^*$, applying an argument similar
to that used in the proof of \cite[Theorem 7.4]{hlyy19}, we conclude that
there exists a unique $g\in\mathcal{L}_{\vp,q',s}^A(\rn)$ such that,
for any $f\in H_{A,\rm fin}^{\vp,q,s}(\rn)$,
$$\Lambda(f) =\int_{\rn}f(x)g(x)\,dx.$$
To finish the proof of (ii), we then only need to show that
$g\in \mathcal{L}_{\vp,q',s,\underline{p}}^A(\rn)$. Indeed,
for any $m\in\nn$, $\{B^{(j)}\}_{j=1}^m\subset \BB$, and
$\{\lambda_j\}_{j=1}^m\subset[0,\infty)$ with $\sum_{j=1}^m\lambda_j\neq0$,
let $\eta_j\in L^q(B^{(j)})$ with $\|\eta_j\|_{L^q(B^{(j)})}=1$ satisfy that
\begin{align}\label{oo11}
\lf[\int_{B^{(j)}}\lf|g(x)-P^s_{B^{(j)}}g(x)\r|^{q'} \,dx\r]^\frac1{q'}
=\int_{B^{(j)}}\lf[g(x)-P^s_{B^{(j)}}g(x)\r]\eta_j(x)\,dx
\end{align}
and, for any $x\in\rn$, define
$$
a_j(x):=\frac{|B^{(j)}|^{\frac{1}{q}}
[\eta_j(x)-P^s_{B^{(j)}}\eta_j(x)]{\mathbf{1}}_{B^{(j)}}}
{\|{\mathbf{1}}_{B^{(j)}}\|_{\lv}
\|\eta_j-P^s_{B^{(j)}}\eta_j\|_{L^q(B^{(j)})}}.
$$
Thus, for any $j\in\{1,\ldots,m\}$, $a_j$ is an $(\vp,q,s)$-atom. By this
and Lemma \ref{3l1}, we know that $\sum_{j=1}^m \lambda_j a_j\in \vh$.
This, together with \eqref{oo11} and the facts that $\Lambda \in(\vh)^*$ and
$$
\lf\|\eta_j-P^s_{B^{(j)}}\eta_j\r\|_{L^q(B^{(j)})}\ls1,
$$
further implies that
\begin{align*}
&\sum_{j=1}^m\frac{{\lambda}_j|B^{(j)}|}{\|{\mathbf{1}}_{B^{(j)}}\|_{\lv}}\lf[
\frac1{|B^{(j)}|}\int_{B^{(j)}}\lf|g(x)-P^s_{B^{(j)}}g(x)\r|^{q'} \,dx\r]^\frac1{q'}\\
&\quad=\sum_{j=1}^m\frac{{\lambda}_j|B^{(j)}|^{\frac1q}}
{\|{\mathbf{1}}_{B^{(j)}}\|_{\lv}}\int_{B^{(j)}}\lf[g(x)-P^s_{B^{(j)}}g(x)\r]\eta_j(x)\,dx\noz\\
&\quad=\sum_{j=1}^m\frac{{\lambda}_j|B^{(j)}|^{\frac1q}}
{\|{\mathbf{1}}_{B^{(j)}}\|_{\lv}}
\int_{B^{(j)}}\lf[\eta_j(x)-P^s_{B^{(j)}}\eta_j(x)\r]g(x){\mathbf{1}}_{B^{(j)}}(x)\,dx\noz\\
&\quad\lesssim
\sum_{j=1}^m{\lambda}_j
\int_{B^{(j)}}a_j(x)g(x)\,dx
\sim \sum_{j=1}^m{\lambda}_j \Lambda(a_j)
\sim \Lambda\lf(\sum_{j=1}^m{\lambda}_j a_j\r)\noz\\
&\quad\ls\lf\|\sum_{j=1}^m{\lambda}_j a_j\r\|_{\vh}
\ls \lf\|\lf\{\sum_{j=1}^m
\lf[\frac{{\lambda}_j}{\|{\mathbf{1}}_{B^{(j)}}\|_{\lv}}\r]^{\underline{p}}
{\mathbf{1}}_{B^{(j)}}\r\}^{\frac1{\underline{p}}}\r\|_{\lv}.\noz
\end{align*}
From this, it follows that $g\in\mathcal{L}_{\vp,q',s,\underline{p}}^A(\rn)$,
which completes the proof of (ii) and hence of Theorem \ref{2t1}.
\end{proof}

\begin{remark}\label{2r3}

We point out that the dual of the anisotropic variable Hardy
space $\vh$ as $\vp\in C^{\log}(\rn)$ with $p_+\in(0,1]$ was
obtained by Wang \cite[Theorem 4.4]{w}. However, using Proposition \ref{2p2},
we find that \cite[Theorem 4.4]{w} is a special case
of Theorem \ref{2t1}. Indeed, note that, in the case
$\vp\in C^{\log}(\rn)$ with $0<p_-\le1<p_+<\fz$,
the dual space of $\vh$ can not be deduced from \cite[Theorem 4.4]{w} anymore.
But, even in this case, the dual of $\vh$ is also contained in Theorem \ref{2t1}.
This is the main contribution of Theorem \ref{2t1}.
Namely, this theorem identifies the new introduced Campanato-type space
$\mathcal{L}_{\vp,q',s,\underline{p}}^A(\rn)$ with the dual space of
the Hardy space $\vh$ for all $\vp\in C^{\log}(\rn)$.

\end{remark}

\begin{remark}\label{2r3'}
Let $\textrm{I}_{n\times n}$ denote the $n\times n$ unit matrix.
When $A:=d\textrm{I}_{n\times n}$ for some $d\in\rr$ with $|d|\in(1,\fz)$, then the space $\vh$
goes back to the variable Hardy space studied in \cite{cw14,ns12}. In this case,
Theorem \ref{2t1} shows the dual theorem of $H^{p(\cdot)}$ with $0<p_-\le p_+<\fz$,
which gives a complete answer to the question proposed by Izuki et al. in
\cite[Section 9.3]{ins13}.

\end{remark}

By Theorem \ref{2t1}, we easily obtain the following equivalence of the
Campanato-type space; we omit the details.

\begin{corollary}\label{2c1}
Let $\vp$, $s$, and $\underline{p}$ be as in Theorem \ref{2t1},
$q\in[1,\fz)$ when $p_+\in(0,1)$, or $q\in[1,p_+')$ when $p_+\in[1,\fz)$.
Then $$\mathcal{L}_{\vp,q,s,\underline{p}}^A(\rn)
=\mathcal{L}_{\vp,1,s_0,p_0}^A(\rn)$$
with equivalent quasi-norms,
where $s_0:=\lfloor(\frac1{p_-}-1)
\frac{\ln b}{\ln\lambda_-}\rfloor$
and $p_0:=\frac{\min\{1,p_-\}}{2}$ with
$p_-$ as in \eqref{2e3}.
\end{corollary}

Let $\vp\in C^{\log}(\rn)$ with $p_-\in(1,\fz)$. Combining Theorem \ref{2t1}, the fact that $(L^{\vp}(\rn))^*=L^{\vp'}(\rn)$
(see \cite[Theorem 2.80]{CUF}), \cite[Corollary 4.20]{zsy},
and Corollary \ref{2c1}, we conclude the following
conclusion; we omit the details.

\begin{corollary}\label{2c2}
Let $\vp\in C^{\log}(\rn)$ with $p_-\in(1,\fz)$,
and $s_0$ and $p_0$ be as in Corollary
\ref{2c1}. Then
$$L^{\vp'}(\rn)=\mathcal{L}_{\vp,1,s_0,p_0}^A(\rn)$$
with equivalent quasi-norms.
\end{corollary}

\section{Applications\label{s3}}

As applications, this section is devoted to establishing
several equivalent characterizations of the new anisotropic
variable Campanato-type spaces and
the Carleson measure characterizations of $\mathcal{L}_{\vp,1,s,\underline{p}}^A({\rn})$
via the dual result obtained in Theorem \ref{2t1}.
To begin with, we show the following theorem.

\begin{theorem}\label{2t2}
Let $\vp,\ q,\ s$, and $\underline{p}$ be as in Corollary \ref{2c1}
and $\vaz\in([2/r-1]\ln b/\ln\lz_-,\fz)$ for some $r\in(0,\underline{p})$.
Then the following statements are mutually equivalent:
\begin{enumerate}
\item[{\rm(i)}]
$f\in\mathcal{L}_{\vp,q,s,\underline{p}}^A(\rn)$;

\item[{\rm(ii)}]
$f\in L^q_{\rm loc}(\rn)$ and
\begin{align*}
\|f\|_{\mathcal{L}_{\vp,1,s,\underline{p}}^A(\rn)}
&\,:=\sup
\lf\|\lf\{\sum_{i=1}^m
\lf[\frac{{\lambda}_i}{\|{\mathbf{1}}_{x_i+B_{l_i}}\|_{\lv}}\r]^{\underline{p}}
{\mathbf{1}}_{x_i+B_{l_i}}\r\}^{\frac1{\underline{p}}}\r\|_{\lv}^{-1}\\
&\quad\times\sum_{j=1}^m\lf\{\frac{{\lambda}_j}{\|{\mathbf{1}}_{x_j+B_{l_j}}
\|_{\lv}}\int_{x_j+B_{l_j}}\lf|f(x)-P^s_{x_j+B_{l_j}}f(x)\r| \,dx\r\}\\
&\,<\infty,
\end{align*}
where the supremum
is taken over all $m\in\nn$, $\{x_j+B_{l_j}\}_{j=1}^m\subset \BB$
with $\{x_j\}_{j=1}^m\subset\rn$ and $\{l_j\}_{j=1}^m\subset\zz$, and
$\{\lambda_j\}_{j=1}^m\subset[0,\infty)$ with $\sum_{j=1}^m\lambda_j\neq0$;

\item[{\rm(iii)}]
$f\in L^q_{\rm loc}(\rn)$ and
\begin{align*}
\||f\||_{\mathcal{L}_{\vp,q,s,\underline{p}}^A(\rn)}
&\,:=\sup\inf
\lf\|\lf\{\sum_{i=1}^m
\lf[\frac{{\lambda}_i}{\|{\mathbf{1}}_{x_i+B_{l_i}}\|_{\lv}}\r]^{\underline{p}}
{\mathbf{1}}_{x_i+B_{l_i}}\r\}^{\frac1{\underline{p}}}\r\|_{\lv}^{-1}\\
&\quad\times\sum_{j=1}^m\lf\{\frac{{\lambda}_j|x_j+B_{l_j}|}{\|{\mathbf{1}}_{x_j+B_{l_j}}
\|_{\lv}}\lf[
\frac1{|x_j+B_{l_j}|}\int_{x_j+B_{l_j}}\lf|f(x)-P(x)\r|^q \,dx\r]^\frac1q\r\}\\
&\,<\fz,
\end{align*}
where the supremum is the same as in (ii) and
the infimum is taken over all $P\in\mathbb{P}_s(\rn)$;

\item[{\rm(iv)}]
$f\in L^q_{\rm loc}(\rn)$ and
\begin{align*}
\|f\|_{\mathcal{L}_{\vp,1,s,\underline{p}}^{A,\vaz}(\rn)}
&\,:=\sup
\lf\|\lf\{\sum_{i=1}^m
\lf[\frac{{\lambda}_i}{\|{\mathbf{1}}_{x_i+B_{l_i}}\|_{\lv}}\r]^{\underline{p}}
{\mathbf{1}}_{x_i+B_{l_i}}\r\}^{\frac1{\underline{p}}}\r\|_{\lv}^{-1}\\
&\quad\times\sum_{j=1}^m\lf\{\frac{{\lambda}_j|x_j+B_{l_j}|}{\|{\mathbf{1}}_{x_j+B_{l_j}}
\|_{\lv}}\int_{\rn}\frac{b^{\vaz l_j\frac{\ln\lz_-}{\ln b}}|f(x)-P^s_{x_j+B_{l_j}}f(x)|}
{b^{l_j(1+\vaz\frac{\ln\lz_-}{\ln b})}+[\rho(x-x_j)]^{1+\vaz\frac{\ln\lz_-}{\ln b}}}\,dx\r\}\\
&\,<\infty,
\end{align*}
where the supremum is the same as in (ii).
\end{enumerate}

Moreover, for any $f\in L^q_{\rm loc}(\rn)$,
$$\|f\|_{\mathcal{L}_{\vp,q,s,\underline{p}}^A(\rn)}\sim
\|f\|_{\mathcal{L}_{\vp,1,s,\underline{p}}^A(\rn)}\sim
\||f\||_{\mathcal{L}_{\vp,q,s,\underline{p}}^A(\rn)}\sim
\|f\|_{\mathcal{L}_{\vp,1,s,\underline{p}}^{A,\vaz}(\rn)}$$
with the positive equivalence constants
independent of $f$.
\end{theorem}

To prove Theorem \ref{2t2}, we need the following inequality,
which is just a consequence of \cite[Lemma 4.4]{lwyy17}.

\begin{lemma}\label{5l2}
Let $\vp\in C^{\log}(\rn)$ and $r\in(0,\min\{1,p_-\})$
with $p_-$ as in \eqref{2e3}. Then there exists a
positive constant $C$ such that, for any
$\{x_j+B_{l_j}\}_{j=1}^m\subset \BB$ with
$\{x_j\}_{j=1}^m\subset\rn$ and $\{l_j\}_{j=1}^m\subset\zz$,
and $k\in\nn$,
\begin{align*}
\lf\|\sum_{j\in\nn}\mathbf{1}_{x_j+B_{l_j+k}}\r\|_{\lv}
\le Cb^{k/r}\lf\|\sum_{j\in\nn}\mathbf{1}_{x_j+B_{l_j}}\r\|_{\lv}.
\end{align*}
\end{lemma}

Now we show Theorem \ref{2t2}.

\begin{proof}[Proof of Theorem \ref{2t2}]
From Corollary \ref{2c1}, we easily deduce the equivalence of (i) and (ii).
Moreover, (i) obviously implies (iii) and
conversely, by an argument similar to the proof of \cite[Proposition 4.1]{lu}, we know
that (iii) implies (i). Therefore, to prove this theorem,
it suffices to show that (ii) is equivalent to (iv).

Assume that (iv) holds true. We now show (ii). Indeed, for any $m\in\nn$,
let $\{x_j+B_{l_j}\}_{j=1}^m\subset \BB$ with
$\{x_j\}_{j=1}^m\subset\rn$ and $\{l_j\}_{j=1}^m\subset\zz$, and $\{\lz_j\}_{j=1}^m\subset [0,\fz)$
with $\sum_{j=1}^m\lambda_j\neq0$.
Then, for any $\vaz\in(0,\fz)$ and $j\in\{1,\ldots,m\}$,
we have
\begin{align*}
\int_{\rn}\frac{b^{\vaz l_j\frac{\ln\lz_-}{\ln b}}|f(x)-P^s_{x_j+B_{l_j}}f(x)|}
{b^{l_j(1+\vaz\frac{\ln\lz_-}{\ln b})}+[\rho(x-x_j)]^{1+\vaz\frac{\ln\lz_-}{\ln b}}}\,dx
&\gs \int_{x_j+B_{l_j}}\frac{b^{\vaz l_j\frac{\ln\lz_-}{\ln b}}|f(x)-P^s_{x_j+B_{l_j}}f(x)|}
{b^{l_j(1+\vaz\frac{\ln\lz_-}{\ln b})}}\,dx\\
&\sim \frac1{|x_j+B_{l_j}|}\int_{x_j+B_{l_j}}\lf|f(x)-P^s_{x_j+B_{l_j}}f(x)\r|\,dx,
\end{align*}
which, together with the assumption
$\|f\|_{\mathcal{L}_{\vp,1,s,\underline{p}}^{A,\vaz}(\rn)}<\fz$, implies
that (ii) holds true. Thus, (iv) implies (ii).

Next assume that (ii) holds true. We prove (iv).
Note that, for any $m\in\nn$, $\{x_j+B_{l_j}\}_{j=1}^m\subset \BB$ with
$\{x_j\}_{j=1}^m\subset\rn$ and $\{l_j\}_{j=1}^m\subset\zz$,
and $\{\lz_j\}_{j=1}^m\subset [0,\fz)$
with $\sum_{j=1}^m\lambda_j\neq0$, we have
\begin{align*}
&\sum_{j=1}^m\lf\{\frac{{\lambda}_j|x_j+B_{l_j}|}{\|{\mathbf{1}}_{x_j+B_{l_j}}
\|_{\lv}}\int_{\rn}\frac{b^{\vaz l_j\frac{\ln\lz_-}{\ln b}}|f(x)-P^s_{x_j+B_{l_j}}f(x)|}
{b^{l_j(1+\vaz\frac{\ln\lz_-}{\ln b})}+[\rho(x-x_j)]
^{1+\vaz\frac{\ln\lz_-}{\ln b}}}\,dx\r\}\\
&\hs=\sum_{j=1}^m\lf\{\frac{{\lambda}_j|x_j+B_{l_j}|}{\|{\mathbf{1}}_{x_j+B_{l_j}}
\|_{\lv}}\lf[\int_{x_j+B_{l_j}}+\sum_{k=0}^{\fz}
\int_{b^{l_j+k}\le\rho(x-x_j)<b^{l_j+k+1}}\r]\r.\\
&\hs\hs\hs\times\lf.\frac{b^{\vaz l_j\frac{\ln\lz_-}{\ln b}}|f(x)-P^s_{x_j+B_{l_j}}f(x)|}
{b^{l_j(1+\vaz\frac{\ln\lz_-}{\ln b})}+[\rho(x-x_j)]
^{1+\vaz\frac{\ln\lz_-}{\ln b}}}\,dx\r\}\\
&\hs\ls\sum_{j=1}^m\lf\{\frac{{\lambda}_j}{\|{\mathbf{1}}_{x_j+B_{l_j}}
\|_{\lv}}\int_{x_j+B_{l_j}}\lf|f(x)-P^s_{x_j+B_{l_j}}f(x)\r| \,dx\r\}\\
&\hs\hs\hs+\sum_{j=1}^m\lf\{\frac{{\lambda}_j}{\|{\mathbf{1}}_{x_j+B_{l_j}}
\|_{\lv}}\sum_{k=0}^{\fz}b^{-k({1+\vaz\frac{\ln\lz_-}{\ln b}})}\int_{b^{l_j+k}\le\rho(x-x_j)<b^{l_j+k+1}}\lf|f(x)-P^s_{x_j+B_{l_j}}f(x)\r|\,dx\r\}.
\end{align*}
Therefore,
\begin{align}\label{5e1}
&\lf\|\lf\{\sum_{i=1}^m
\lf[\frac{{\lambda}_i}{\|{\mathbf{1}}_{x_i+B_{l_i}}\|_{\lv}}\r]^{\underline{p}}
{\mathbf{1}}_{x_i+B_{l_i}}\r\}^{\frac1{\underline{p}}}\r\|_{\lv}^{-1}\noz\\
&\quad\quad\times\sum_{j=1}^m\lf\{\frac{{\lambda}_j|x_j+B_{l_j}|}{\|{\mathbf{1}}_{x_j+B_{l_j}}
\|_{\lv}}\int_{\rn}\frac{b^{\vaz l_j\frac{\ln\lz_-}{\ln b}}|f(x)-P^s_{x_j+B_{l_j}}f(x)|}
{b^{l_j(1+\vaz\frac{\ln\lz_-}{\ln b})}+[\rho(x-x_j)]
^{1+\vaz\frac{\ln\lz_-}{\ln b}}}\,dx\r\}\noz\\
&\hs\ls \|f\|_{\mathcal{L}_{\vp,1,s,\underline{p}}^A(\rn)}
+\Theta,
\end{align}
where
\begin{align*}
\Theta:=&\lf\|\lf\{\sum_{i=1}^m
\lf[\frac{{\lambda}_i}{\|{\mathbf{1}}_{x_i+B_{l_i}}\|_{\lv}}\r]^{\underline{p}}
{\mathbf{1}}_{x_i+B_{l_i}}\r\}^{\frac1{\underline{p}}}\r\|_{\lv}^{-1}\\
&\times \sum_{j=1}^m\lf\{\frac{{\lambda}_j}{\|{\mathbf{1}}_{x_j+B_{l_j}}
\|_{\lv}}\sum_{k=0}^{\fz}b^{-k({1+\vaz\frac{\ln\lz_-}{\ln b}})}\int_{b^{l_j+k}\le\rho(x-x_j)
<b^{l_j+k+1}}\lf|f(x)-P^s_{x_j+B_{l_j}}f(x)\r| \,dx\r\}\\
\ls&\lf\|\lf\{\sum_{i=1}^m
\lf[\frac{{\lambda}_i}{\|{\mathbf{1}}_{x_i+B_{l_i}}\|_{\lv}}\r]^{\underline{p}}
{\mathbf{1}}_{x_i+B_{l_i}}\r\}^{\frac1{\underline{p}}}\r\|_{\lv}^{-1}\\
&\times \sum_{j=1}^m\lf\{\frac{{\lambda}_j}{\|{\mathbf{1}}_{x_j+B_{l_j}}
\|_{\lv}}\sum_{k\in\nn}b^{-k({1+\vaz\frac{\ln\lz_-}{\ln b}})}\int_{x_j+B_{l_j+k}}\lf|f(x)-P^s_{x_j+B_{l_j+k}}f(x)\r| \,dx\r\}\\
&+\lf\|\lf\{\sum_{i=1}^m
\lf[\frac{{\lambda}_i}{\|{\mathbf{1}}_{x_i+B_{l_i}}\|_{\lv}}\r]^{\underline{p}}
{\mathbf{1}}_{x_i+B_{l_i}}\r\}^{\frac1{\underline{p}}}\r\|_{\lv}^{-1}\\
&\times \sum_{j=1}^m\lf\{\frac{{\lambda}_j}{\|{\mathbf{1}}_{x_j+B_{l_j}}
\|_{\lv}}\sum_{k\in\nn}b^{-k({1+\vaz\frac{\ln\lz_-}{\ln b}})}\int_{x_j+B_{l_j+k}}\lf|P^s_{x_j+B_{l_j+k}}f(x)-P^s_{x_j+B_{l_j}}f(x)\r| \,dx\r\}.
\end{align*}
This, together with the Tonelli theorem, $r\in(0,\underline{p})$, and Lemma \ref{5l2},
implies that
\begin{align}\label{5e2}
\Theta
&\ls \|f\|_{\mathcal{L}_{\vp,1,s,\underline{p}}^A(\rn)}
\sum_{k\in\nn}b^{-k({1-\frac {2}{r}+\vaz\frac{\ln\lz_-}{\ln b}})}
+\lf\|\lf\{\sum_{i=1}^m
\lf[\frac{{\lambda}_i}{\|{\mathbf{1}}_{x_i+B_{l_i}}\|_{\lv}}\r]^{\underline{p}}
{\mathbf{1}}_{x_i+B_{l_i}}\r\}^{\frac1{\underline{p}}}\r\|_{\lv}^{-1}\noz\\
&\quad\times \sum_{j=1}^m\lf\{\frac{{\lambda}_j}{\|{\mathbf{1}}_{x_j+B_{l_j}}
\|_{\lv}}\sum_{k\in\nn}b^{-k({1+\vaz\frac{\ln\lz_-}{\ln b}})}\int_{x_j+B_{l_j+k}}\lf|P^s_{x_j+B_{l_j+k}}f(x)
-P^s_{x_j+B_{l_j}}f(x)\r| \,dx\r\}.
\end{align}
On another hand, from \cite[(8.9)]{mb03}, we deduce that,
for any $x\in x_j+B_{l_j+k}$,
\begin{align*}
\lf|P^s_{x_j+B_{l_j+k}}f(x)-P^s_{x_j+B_{l_j}}f(x)\r|
&\le \sum_{\nu=1}^k\lf|P^s_{x_j+B_{l_j+\nu}}f(x)-P^s_{x_j+B_{l_j+\nu-1}}f(x)\r|\\
&=\sum_{\nu=1}^k\lf|P^s_{x_j+B_{l_j+\nu-1}}\lf(f-P^s_{x_j+B_{l_j+\nu}}f\r)(x)\r|\\
&\ls \sum_{\nu=1}^k \frac1{|x_j+B_{l_j+\nu-1}|}
\int_{x_j+B_{l_j+\nu}}\lf|f(x)-P^s_{x_j+B_{l_j+\nu}}f(x)\r|\,dx,
\end{align*}
which, combined with \eqref{5e2}, the Tonelli theorem, Lemma \ref{5l2},
and the fact that $\vaz\in([2/r-1]\ln b/\ln\lz_-,\fz)$,
further implies that
\begin{align*}
\Theta&\ls \|f\|_{\mathcal{L}_{\vp,1,s,\underline{p}}^A(\rn)}
\sum_{k\in\nn}b^{-k({1-\frac {2}{r}+\vaz\frac{\ln\lz_-}{\ln b}})}
+\lf\|\lf\{\sum_{i=1}^m
\lf[\frac{{\lambda}_i}{\|{\mathbf{1}}_{x_i+B_{l_i}}\|_{\lv}}\r]^{\underline{p}}
{\mathbf{1}}_{x_i+B_{l_i}}\r\}^{\frac1{\underline{p}}}\r\|_{\lv}^{-1}\\
&\quad\times \sum_{j=1}^m\lf\{\frac{{\lambda}_j}{\|{\mathbf{1}}_{x_j+B_{l_j}}
\|_{\lv}}\sum_{k\in\nn}b^{-k\vaz\frac{\ln\lz_-}{\ln b}}\sum_{\nu=1}^k b^{-\nu}\int_{x_j+B_{l_j+\nu}}\lf|f(x)-P^s_{x_j+B_{l_j+\nu}}f(x)\r| \,dx\r\}\\
&\ls \|f\|_{\mathcal{L}_{\vp,1,s,\underline{p}}^A(\rn)}
\sum_{k\in\nn}b^{-k({1-\frac {2}{r}+\vaz\frac{\ln\lz_-}{\ln b}})}
+\|f\|_{\mathcal{L}_{\vp,1,s,\underline{p}}^A(\rn)}
\sum_{k\in\nn}b^{-k\vaz\frac{\ln\lz_-}{\ln b}}\sum_{\nu=1}^k b^{\nu(2/r-1)}\\
&\ls \|f\|_{\mathcal{L}_{\vp,1,s,\underline{p}}^A(\rn)}
\sum_{k\in\nn}b^{-k({1-\frac {2}{r}+\vaz\frac{\ln\lz_-}{\ln b}})}
\ls\|f\|_{\mathcal{L}_{\vp,1,s,\underline{p}}^A(\rn)}.
\end{align*}
Combining this and \eqref{5e1}, we finally conclude that (iv) holds
true and hence finish the proof of Theorem \ref{2t2}.
\end{proof}

From Proposition \ref{2p1} and Theorem \ref{2t2}, we immediately deduce the
following equivalent characterizations of $\mathcal{L}_{\vp,q,s,\underline{p}}^A(\rn)$;
we omit the details.

\begin{corollary}\label{2c3}
If $\vp,\ q,\ s$, $\underline{p}$, and $\vaz$ are as in Theorem \ref{2t2},
then all the conclusions of Theorem \ref{2t2} still hold true with $m$ replaced by $\fz$,
and the supremum therein taken over all $\{x_j+B_{l_j}\}_{j\in\nn}\subset \BB$
with $\{x_j\}_{j\in\nn}\subset\rn$ and $\{l_j\}_{j\in\nn}\subset\zz$, and
$\{\lambda_j\}_{j\in\nn}\subset[0,\infty)$ satisfying
\begin{align*}
\lf\|\lf\{\sum_{j\in\nn}
\lf[\frac{{\lambda}_j}{\|{\mathbf{1}}_{x_j+B_{l_j}}\|_\lv}\r]^{\underline{p}}
{\mathbf{1}}_{x_j+B_{l_j}}\r\}^{\frac1{\underline{p}}}\r\|_{\lv}\in(0,\infty).
\end{align*}
\end{corollary}

Via these equivalent characterizations of the anisotropic
variable Campanato-type space, we next establish the Carleson measure
characterization of $\mathcal{L}_{\vp,1,s,\underline{p}}^A({\rn})$. To
begin with, we introduce the $\vp$-Carleson measure as follows.

\begin{definition}\label{4d1}
Let $\vp\in \mathcal{P}(\rn)$. A Borel measure $d\mu$ on $\rn\times\zz$
is called a \emph{$\vp$-Carleson measure} if
\begin{align*}
\lf\|d\mu\r\|_{\mathcal{C}_{\vp,A}}:=\sup
\lf\|\lf\{\sum_{i=1}^m
\lf[\frac{{\lambda}_i}{\|{\mathbf{1}}_{B^{(i)}}\|_{\lv}}\r]^{\eta}
{\mathbf{1}}_{B^{(i)}}\r\}^{1/{\eta}}\r\|_{\lv}^{-1}\sum_{j=1}^m\frac{{\lambda}_j|B^{(j)}|^{1/2}}{\|{\mathbf{1}}_{B^{(j)}}
\|_{\lv}}\lf[\int_{\widehat{B^{(j)}}}\,|d\mu(x,k)|\r]^{1/2}
\end{align*}
is finite, where $\eta\in(0,\fz)$ and the supremum
is taken over all $m\in\nn$, $\{B^{(j)}\}_{j=1}^m\subset \BB$, and
$\{\lambda_j\}_{j=1}^m\subset[0,\infty)$ with $\sum_{j=1}^m\lambda_j\neq0$,
and, for any $j\in\{1,\ldots,m\}$,
$\widehat{B^{(j)}}$ denotes the \emph{tent} over $B^{(j)}$,
namely,
$$\widehat{B^{(j)}}:=\lf\{(y,k)\in\rn\times\zz:\ y+B_k\subset B^{(j)}\r\}.$$
\end{definition}

\begin{remark}\label{4r1}
Let $\vp$, $\eta$, and $d\mu$ be as in Definition \ref{4d1} and
\begin{align*}
\widetilde{\lf\|d\mu\r\|}_{\mathcal{C}_{\vp,A}}:=\sup
\lf\|\lf\{\sum_{i\in\nn}
\lf[\frac{{\lambda}_i}{\|{\mathbf{1}}_{B^{(i)}}\|_{\lv}}\r]^{\eta}
{\mathbf{1}}_{B^{(i)}}\r\}^{1/{\eta}}\r\|_{\lv}^{-1}
\sum_{j\in\nn}\frac{{\lambda}_j|B^{(j)}|^{1/2}}{\|{\mathbf{1}}_{B^{(j)}}
\|_{\lv}}\lf[\int_{\widehat{B^{(j)}}}\,|d\mu(x,k)|\r]^{1/2},
\end{align*}
where the supremum is taken over all
$\{B^{(j)}\}_{j\in\nn}\subset \BB$ and
$\{\lambda_j\}_{j\in\nn}\subset[0,\infty)$ satisfying
\begin{align*}
\lf\|\lf\{\sum_{j\in\nn}
\lf[\frac{{\lambda}_j}{\|{\mathbf{1}}_{B^{(j)}}\|_{\lv}}\r]^{\eta}
{\mathbf{1}}_{B^{(j)}}\r\}^{1/\eta}\r\|_{\lv}\in(0,\infty).
\end{align*}
Then $\widetilde{\lf\|d\mu\r\|}_{\mathcal{C}_{\vp,A}}
=\lf\|d\mu\r\|_{\mathcal{C}_{\vp,A}}$.
\end{remark}

In what follows, for any given $k\in\zz$, define
\begin{align*}
\delta_k(j):=\left\{
\begin{array}{cl}
\vspace{0.15cm}
1 &{\rm when}~~ j=k,\\
0 &{\rm when}~~ j\neq k.
\end{array}\r.
\end{align*}
Let $C_{\rm c}^{\fz}(\rn)$ denote the set of all
\emph{infinitely differentiable functions with compact support}
on $\rn$ and, for any $\varphi\in\cs(\rn)$,
$\widehat{\varphi}$ denote its \emph{Fourier transform},
namely, for any $\xi\in\rn$,
\begin{align*}
\widehat \varphi(\xi) := \int_{\rn} \varphi(x) e^{-2\pi\imath x \cdot \xi} \, dx,
\end{align*}
where $\imath:=\sqrt{-1}$ and $x\cdot\xi :=\sum_{i=1}^{n}x_i \xi_i$
for any $x:=(x_1,\ldots,x_n)$, $\xi:=(\xi_1,\ldots,\xi_n) \in \rn$.
Let $s\in\zz_+$ and $\phi\in C_{\rm c}^{\fz}(\rn)$ satisfy
\begin{align}\label{4e1'}
\supp\phi\subset B_0,~~~
\int_{\rn}x^\gamma\phi(x)\,dx=0,~~~\forall~\gamma\in\zz_+^n~~
{\rm with}~~|\gamma|\le s,
\end{align}
and there exists a positive constant $C$ such that
\begin{align}\label{4e2}
|\widehat{\phi}(\xi)|\ge C
~~~{\rm when}~~~\xi\in\lf\{x\in\rn:\ (2\|A\|)^{-1}\le\rho(x)\le 1\r\},
\end{align}
where the dilation $A:=(a_{i,j})_{1\le i,j\le n}$ and $\|A\|:=(\sum_{i,j=1}^n|a_{i,j}|^2)^{1/2}$.

We now show the Carleson measure characterization of
$\mathcal{L}_{\vp,1,s,\underline{p}}^A({\rn})$ as follows.

\begin{theorem}\label{2t3}
Let $\vp$, $s$, and $\underline{p}$ be as in Theorem \ref{2t1},
$p_+\in(0,2)$, and $\phi\in\cs(\rn)$ be
a radial real-valued function satisfying \eqref{4e1'} and \eqref{4e2}.
\begin{enumerate}
\item[{\rm (i)}]
If $b\in \mathcal{L}_{\vp,1,s,\underline{p}}^A({\rn})$, then,
for any $(x,k)\in\rn\times\zz$, $d\mu(x,k):=\sum_{\ell\in\zz}
|\phi_{-\ell}\ast b(x)|^2\,dx\,\delta_{\ell}(k)$ is a
$\vp$-Carleson measure on $\rn\times\zz$;
moreover, there exists a positive constant $C$, independent of $b$, such that
$\|d\mu\|_{\mathcal{C}_{\vp,A}}
\le C\|b\|_{\mathcal{L}_{\vp,1,s,\underline{p}}^A({\rn})}$.

\item[{\rm (ii)}]
If $b\in L^2_{\rm loc}(\rn)$ and, for any
$(x,k)\in\rn\times\zz$, $d\mu(x,k):=\sum_{\ell\in\zz}
|\phi_{-\ell}\ast b(x)|^2\,dx\,\delta_{\ell}(k)$ is
a $\vp$-Carleson measure on $\rn\times\zz$,
then $b\in\mathcal{L}_{\vp,1,s,\underline{p}}^A({\rn})$
and, moreover, there exists a positive
constant $C$, independent of $b$, such that
$\|b\|_{\mathcal{L}_{\vp,1,s,\underline{p}}^A({\rn})}
\le C\|d\mu\|_{\mathcal{C}_{\vp,A}}$.
\end{enumerate}
\end{theorem}

To prove the above Carleson measure characterization, we need
the anisotropic variable tent space and its atomic decomposition.
However, as far as we know, the theory about anisotropic variable tent spaces is blank
in the literature. Thus, we first introduce the anisotropic variable tent space
and then establish its atomic decomposition.
For any $x\in\rn$, let
$$\Gamma(x):=\{(y,k)\in\rn\times\zz:\ y\in x+B_k\}$$
be the \emph{cone} of aperture $1$ with vertex $x\in\rn$.
For any measurable function $G$ on $\rn\times\zz$,
the \emph{anisotropic discrete Lusin area function $\ca(G)$} is defined
by setting, for any $x\in\rn$,
\begin{equation*}
\ca (G)(x):=\lf[\sum_{\ell\in\zz}b^{-\ell}
\int_{\{y\in\rn:\ (y,\ell)\in\Gamma(x)\}}|G(y,\ell)|^2\,
dy\r]^\frac12.
\end{equation*}

Via this anisotropic discrete Lusin area function, we introduce the anisotropic
variable tent space as follows.

\begin{definition}\label{d1}
Let $\vp\in \mathcal{P}(\rn)$. The \emph{anisotropic variable tent space}
$T^{\vp}_A(\rn\times\zz)$ is defined to be the set of all measurable
functions $G$ on $\rn\times\zz$ such that $\ca (G)\in \lv$
and equipped with the quasi-norm
$\|G\|_{T^{\vp}_A(\rn\times\zz)}:=\|\ca (G)\|_{\lv}$.
\end{definition}

We next give the definition of anisotropic $(\vp,\fz)$-atoms.

\begin{definition}\label{1q1}
Let $\vp\in \mathcal{P}(\rn)$ and $q\in(1,\fz)$.
A measurable function $a$ on $\rn\times\zz$ is called an
\emph{anisotropic $(\vp,q)$-atom} if there
exists a ball $B\in\BB$ such that
\begin{itemize}
\item[(i)] $\supp a:=\{(x,k)\in\rn\times\zz:\ a(x,k)\neq0\}\subset
\widehat{B}$,
\item[(ii)] $\|a\|_{T_2^{q}(\rn\times\zz)}
:=\|\ca (a)\|_{L^q(\rn)}\le\frac{|B|^{1/q}}{\|\mathbf{1}_B\|_{\lv}}$.
\end{itemize}
Moreover, if $a$ is an anisotropic $(\vp,q)$-atom for any $q\in(1,\infty)$,
then $a$ is called an \emph{anisotropic $(\vp,\fz)$-atom}.
\end{definition}

For functions in the anisotropic variable tent space $T^{\vp}_A(\rn\times\zz)$,
we have the following atomic decomposition.

\begin{lemma}\label{tent}
Let $\vp\in C^{\log}(\rn)$. Then, for any
$G\in T^{\vp}_A(\rn\times\zz)$, there exist
$\{\lambda_j\}_{j\in\nn}\subset [0,\infty)$,
$\{B^{(j)}\}_{j\in\nn}\subset\BB$,
and a sequence $\{A_j\}_{j\in\nn}$ of anisotropic $(\vp,\infty)$-atoms
supported, respectively, in
$\{\widehat{B^{(j)}}\}_{j\in\nn}$ such that, for almost every
$(x,k)\in\rn\times\zz$,
\begin{equation*}
G(x,k)=\sum_{j\in\nn}\lambda_jA_j(x,k)\quad\text{and}\quad|G(x,k)|
=\sum_{j\in\nn}\lambda_j|A_j(x,k)|
\end{equation*}
pointwisely, and
\begin{align}\label{3e3}
\lf\|\lf\{\sum_{j\in\nn}\lf[\frac{\lambda_j}
{\|\mathbf1_{B^{(j)}}\|_{\lv}}\r]^{\underline{p}}
\mathbf1_{B^{(j)}}\r\}^{1/\underline{p}}\r\|_{\lv}
\ls \|G\|_{T^{\vp}_A(\rn\times\zz)},
\end{align}
where the implicit positive constant is independent of $G$.
\end{lemma}

\begin{proof}
Let $\vp\in C^{\log}(\rn)$ and $G\in T^{\vp}_A(\rn\times\zz)$.
For any $j\in\zz$, let
$O_j:=\{x\in\rn:\ \ca(G)(x)>2^j\}$, $G_j:=(O_j)^{\com}$, and,
for any given $\gamma\in(0,1),$
$$(O_j)^*_{\gamma}:=\lf\{x\in\rn:\ \HL(\mathbf{1}_{O_j})(x)>1-\gamma\r\},$$
here and thereafter, $\HL$ denotes the anisotropic Hardy--Littlewood maximal operator,
namely, for any $f\in L^1_{{\rm loc}}(\rn)$ and $x\in\rn$,
\begin{align*}
M_{{\rm HL}}(f)(x):=\sup_{k\in\zz}
\sup_{y+B_k\ni x}\frac1{|B_k|}
\int_{y+B_k}|f(z)|\,dz.
\end{align*}
Then, by the proof of \cite[(1.14)]{fl}, we find that
\begin{align}\label{3e1`}
\supp G\subset\lf[\bigcup_{j\in\zz}\widehat{(O_j)^*_{\gamma}}\cup E\r],
\end{align}
where $E\subset\rn\times\zz$ satisfies that
$\sum_{\ell\in\zz}\int_{\{y\in\rn:\ (y,\ell)\in E\}}\,dy=0$.
Moreover, applying \cite[(1.15)]{fl}, we know that,
for any $j\in\zz$, there exist an integer $N_j\in\nn\cup\{\fz\}$,
$\{x_k^j\}_{k=1}^{N_j}\subset (O_j)^*_{\gamma}$, and
$\{l_k\}_{k=1}^{N_j}\subset\zz$ such that
$\{x_k^j+B_{l_k}^j\}_{k=1}^{N_j}$ has finite intersection property
and
\begin{align}\label{3e4}
(O_j)^*_{\gamma}&=\bigcup_{k=1}^{N_j} \lf(x_k^j+B_{l_k}^j\r)\noz\\
&=\lf(x_1^j+B_{l_1}^j\r)\cup\lf\{\lf(x_2^j+B_{l_2}^j\r)\setminus\lf(x_1^j+B_{l_1}^j\r)\r\}
\cup\cdots\cup\lf\{\lf(x_{N_j}^j+B_{l_{N_j}}^j\r)\setminus \bigcup_{i=1}^{N_j-1}\lf(x_i^j+B_{l_i}^j\r)\r\}\noz\\
&=:\bigcup_{k=1}^{N_j} B_{j,k}.
\end{align}
Note that, for any $j\in\zz$, $\{B_{j,k}\}_{k=1}^{N_j}$ are mutually disjoint.
Thus, $\widehat{(O_j)^*_{\gamma}}=\bigcup_{k=1}^{N_j} \widehat{B_{j,k}}.$
For any $j\in\zz$ and $k\in\{1,\ldots,N_j\}$, let
$C_{j,k}:=\widehat{B_{j,k}}\cap\lf[\widehat{(O_j)^*_{\gamma}}
\setminus \widehat{(O_{j+1})^*_{\gamma}}\r],$
\begin{align}\label{3e2}
A_{j,k}:=2^{-j}\lf\|\mathbf{1}_{x_k^j+B_{l_k}^j}\r\|_{\lv}^{-1}G\mathbf{1}_{C_{j,k}},
\end{align}
and $\lz_{j,k}:=2^j\|\mathbf{1}_{x_k^j+B_{l_k}^j}\|_{\lv}$.
Therefore, from \eqref{3e1`}, it follows that
$$G=\sum_{j\in\zz}\sum_{k=1}^{N_j}
\lz_{j,k}A_{j,k}~~~~{\rm and}~~~~|G|=\sum_{j\in\zz}\sum_{k=1}^{N_j}
\lz_{j,k}|A_{j,k}|$$ almost everywhere on $\rn\times\zz$.
We now show that, for any $j\in\zz$ and $k\in\{1,\ldots,N_j\}$,
$A_{j,k}$ is an anisotropic $(\vp,\fz)$-atom supported in $\widehat{x_k^j+B_{l_k}^j}$.
Obviously,
$\supp A_{j,k}\subset C_{j,k}\subset \widehat{B_{j,k}}\subset \widehat{x_k^j+B_{l_k}^j}.$
In addition, let $q\in(1,\fz)$ and
$h\in T_2^{q'}(\rn\times\zz)$ satisfy $\|h\|_{T_2^{q'}(\rn\times\zz)}\le 1$.
Note that $C_{j,k}\subset \widehat{(O_{j+1})^*_{\gamma}}^{\com}
=\bigcup_{x\in (G_{j+1})^*_{\gamma}}\Gamma(x).$ Applying this,
\cite[Lemma 1.3]{fl}, the H\"{o}lder inequality, and \eqref{3e2},
we find that
\begin{align*}
\lf|\lf\langle A_{j,k},h\r\rangle\r|
&=\lf|\sum_{\ell\in\zz}\int_{C_{j,k}}A_{j,k}(y,\ell)h(y,\ell)\,dy\r|\\
&\ls\int_{G_{j+1}}\sum_{\ell\in\zz}\int_{\{y\in\rn:\
(y,\ell)\in\Gamma(x)\}}b^{-\ell}\lf|A_{j,k}(y,\ell)h(y,\ell)\r|\,dy\,dx\\
&\ls\int_{(O_{j+1})^{\com}}\ca(A_{j,k})(x)\ca(h)(x)\,dx\\
&\ls\lf\{\int_{(O_{j+1})^{\com}}\lf[\ca(A_{j,k})(x)\r]^q\,dx\r\}^{1/q}
\lf\{\int_{(O_{j+1})^{\com}}\lf[\ca(h)(x)\r]^{q'}\,dx\r\}^{1/{q'}}\\
&\ls 2^{-j}\lf\|\mathbf{1}_{x_k^j+B_{l_k}^j}\r\|_{\lv}^{-1}
\lf\{\int_{(x_k^j+B_{l_k}^j)\cap(O_{j+1})^{\com}}\lf[\ca(G)(x)\r]^q\,dx\r\}^{1/q}
\|h\|_{T_2^{q'}(\rn\times\zz)}\\
&\ls \frac{|x_k^j+B_{l_k}^j|^{1/q}}{\|\mathbf{1}_{x_k^j+B_{l_k}^j}\|_{\lv}},
\end{align*}
which, combined with
$(T_2^{q}(\rn\times\zz))^*=T_2^{q'}(\rn\times\zz)$
(see \cite{cms85,fl}), further
implies that $\|A_{j,k}\|_{T_2^{q}(\rn\times\zz)}
\ls \frac{|x_k^j+B_{l_k}^j|^{1/q}}{\|\mathbf{1}_{x_k^j+B_{l_k}^j}\|_{\lv}}$.
This implies that, for any $j\in\zz$ and $k\in\{1,\ldots,N_j\}$,
$A_{j,k}$ is an anisotropic $(\vp,q)$-atom up to a harmless constant
multiple for all $q\in(1,\fz)$. Thus, for any $j\in\zz$ and $k\in\{1,\ldots,N_j\}$,
$A_{j,k}$ is an anisotropic $(\vp,\fz)$-atom up to a harmless constant multiple.

We next prove \eqref{3e3}. To achieve this, from \eqref{3e4},
 the finite intersection property of
$\{x_k^j+B_{l_k}^j\}_{k=1}^{N_j}$, the fact that
$\mathbf1_{(O_j)^*_{\gamma}}\ls [\HL(\mathbf1_{O_j})]^{1/\tau}$
with $\tau\in(0,\underline{p})$, and \cite[Lemma 4.4]{lwyy17},
we deduce that
\begin{align*}
&\lf\|\lf\{\sum_{j\in\zz}\sum_{k=1}^{N_j}\lf[\frac{\lambda_{j,k}}
{\|\mathbf1_{x_k^j+B_{l_k}^j}\|_{\lv}}\r]^{\underline{p}}
\mathbf1_{x_k^j+B_{l_k}^j}\r\}^{1/\underline{p}}\r\|_{\lv}\\
&\hs=\lf\|\lf\{\sum_{j\in\zz}\sum_{k=1}^{N_j}\lf(2^j
\mathbf1_{x_k^j+B_{l_k}^j}\r)^{\underline{p}}
\r\}^{1/\underline{p}}\r\|_{\lv}
\ls\lf\|\lf\{\sum_{j\in\zz}\lf[2^j
\mathbf1_{(O_j)^*_{\gamma}}\r]^{\underline{p}}
\r\}^{1/\underline{p}}\r\|_{\lv}\\
&\hs\ls\lf\|\lf\{\sum_{j\in\zz}\lf(2^j
\mathbf1_{O_j}\r)^{\underline{p}}
\r\}^{1/\underline{p}}\r\|_{\lv}
\sim \lf\|\lf\{\sum_{j\in\zz}\lf(2^j
\mathbf1_{O_j\setminus O_{j+1}}\r)^{\underline{p}}
\r\}^{1/\underline{p}}\r\|_{\lv}\\
&\hs\sim \lf\|\ca(G)\lf[\sum_{j\in\zz}
\mathbf1_{O_j\setminus O_{j+1}}
\r]^{1/\underline{p}}\r\|_{\lv}
\sim \lf\|\ca(G)\r\|_{\lv}
\sim \|G\|_{T^{\vp}_A(\rn\times\zz)}.
\end{align*}
This implies that \eqref{3e3} holds true and hence finishes
the proof of Lemma \ref{tent}.
\end{proof}

With the help of Theorems \ref{2t1} and \ref{2t2} and Lemma \ref{tent},
we now show Theorem \ref{2t3}.

\begin{proof}[Proof of Theorem \ref{2t3}]
Let $\vp$, $s$, $\underline{p}$, and $\phi$ be as in the present theorem.
We first prove (i). To this end, let $b\in \mathcal{L}_{\vp,1,s,\underline{p}}^A({\rn})$
and $\{x_j+B_{l_j}\}_{j=1}^m\subset \BB$, where
$m\in\nn$, $\{x_j\}_{j=1}^m\subset\rn$, and
$\{l_j\}_{j=1}^m\subset\zz$. Then, for any $j\in\{1,\ldots,m\}$,
we have
\begin{align}\label{4e3}
b&=P^s_{x_j+B_{l_j}}b+\lf(b-P^s_{x_j+B_{l_j}}b\r)\mathbf{1}_{x_j+B_{l_j+\omega}}+
\lf(b-P^s_{x_j+B_{l_j}}b\r)\mathbf{1}_{(x_j+B_{l_j+\omega})^{\com}}\noz\\
&=:b_{j}^{(1)}+b_{j}^{(2)}+b_{j}^{(3)}
\end{align}
with $\omega$ as in \eqref{2e2}.
Note that, for any $\az\in\zz_+^n$ with $|\az|\le s$,
$\int_{\rn}\phi(x)x^{\az}\,dx=0$. Therefore, for any $k\in\zz$
and $j\in\{1,\ldots,m\}$,
$\phi_k\ast b_{j}^{(1)}\equiv 0$ and hence, for any $j\in\{1,\ldots,m\}$,
we have
\begin{equation}\label{4e4}
\sum_{\ell\in\zz}\int_{\{x\in\rn:\ (x,\ell)\in\widehat{x_j+B_{l_j}}\}}\lf|
\phi_{-\ell}\ast b_{j}^{(1)}(x)\r|^2\,dx=0.
\end{equation}
In addition, from the Tonelli theorem and the boundedness
of the $g$-function (see, for instance, \cite[Theorem 6.3]{hlyy19}),
we deduce that, for any $j\in\{1,\ldots,m\}$,
\begin{align}\label{4e10}
&\sum_{\ell\in\zz}\int_{\{x\in\rn:\ (x,\ell)\in\widehat{x_j+B_{l_j}}\}}\lf|
\phi_{-\ell}\ast b_{j}^{(2)}(x)\r|^2\,dx\noz\\
&\hs\le \int_{\rn}\sum_{\ell\in\zz}\lf|
\phi_{-\ell}\ast b_{j}^{(2)}(x)\r|^2\,dx
\ls\lf\|b_{j}^{(2)}\r\|_{L^2(\rn)}^2
\sim \int_{x_j+B_{l_j+\omega}}\lf|b(x)-P^s_{x_j+B_{l_j}}b(x)\r|^2\,dx\noz\\
&\hs\ls \int_{x_j+B_{l_j+\omega}}\lf|b(x)-P^s_{x_j+B_{l_j+\omega}}b(x)\r|^2\,dx
+\int_{x_j+B_{l_j+\omega}}\lf|P^s_{x_j+B_{l_j+\omega}}
b(x)-P^s_{x_j+B_{l_j}}b(x)\r|^2\,dx.
\end{align}
Moreover, by \cite[(8.9)]{mb03} (see also \cite[Lemma 4.1]{lu}),
we know that, for any $x\in x_j+B_{l_j+\omega}$,
\begin{align*}
\lf|P^s_{x_j+B_{l_j+\omega}}b(x)-P^s_{x_j+B_{l_j}}b(x)\r|
&=\lf|P^s_{x_j+B_{l_j}}\lf(b-P^s_{x_j+B_{l_j+\omega}}b\r)(x)\r|\\
&\ls \frac1{|x_j+B_{l_j}|}\int_{x_j+B_{l_j+\omega}}
\lf|b(y)-P^s_{x_j+B_{l_j+\omega}}b(y)\r|\,dy,
\end{align*}
which, together with \eqref{4e10} and Lemma \ref{5l2},
further implies that, for any $m\in\nn$, $\{x_j+B_{l_j}\}_{j=1}^m\subset \BB$,
and $\{\lz_j\}_{j=1}^m\subset [0,\fz)$ with $\sum_{j=1}^m\lambda_j\neq0$,
\begin{align*}
&\lf\|\lf\{\sum_{i=1}^m
\lf[\frac{{\lambda}_i}{\|{\mathbf{1}}_{x_i+B_{l_i}}\|_{\lv}}\r]^{\underline{p}}
{\mathbf{1}}_{x_i+B_{l_i}}\r\}^{1/{\underline{p}}}\r\|_{\lv}^{-1}
\sum_{j=1}^m\frac{{\lambda}_j|x_j+B_{l_j}|^{\frac12}}{\|{\mathbf{1}}_{x_j+B_{l_j}}
\|_{\lv}}\\
&\hs\hs\hs\times\lf[\sum_{\ell\in\zz}\int_{\{x\in\rn:\ (x,\ell)\in\widehat{x_j+B_{l_j}}\}}\lf|
\phi_{-\ell}\ast b_{j}^{(2)}(x)\r|^2\,dx\r]^{\frac12}\\
&\hs\ls \lf\|\lf\{\sum_{i=1}^m
\lf[\frac{{\lambda}_i}{\|{\mathbf{1}}_{x_i+B_{l_i+\omega}}\|_{\lv}}\r]^{\underline{p}}
{\mathbf{1}}_{x_i+B_{l_i+\omega}}\r\}^{1/{\underline{p}}}\r\|_{\lv}^{-1}
\sum_{j=1}^m\frac{{\lambda}_j|x_j
+B_{l_j+\omega}|^{\frac12}}{\|{\mathbf{1}}_{x_j+B_{l_j+\omega}}\|_{\lv}}\\
&\hs\hs\hs\times\lf\{\lf[\int_{x_j+B_{l_j+\omega}}
\lf|b(x)-P^s_{x_j+B_{l_j+\omega}}b(x)\r|^2\,dx\r]^{\frac12}\r.\\
&\hs\hs\hs+\lf.
\frac1{|x_j+B_{l_j}|^{\frac12}}\int_{x_j+B_{l_j+\omega}}
\lf|b(x)-P^s_{x_j+B_{l_j+\omega}}b(x)\r|\,dx\r\}.
\end{align*}
This, combined with $p_+\in(0,2)$ and Corollary \ref{2c1}, further implies that
\begin{align}\label{4e6}
&\lf\|\lf\{\sum_{i=1}^m
\lf[\frac{{\lambda}_i}{\|{\mathbf{1}}_{x_i+B_{l_i}}\|_{\lv}}\r]^{\underline{p}}
{\mathbf{1}}_{x_i+B_{l_i}}\r\}^{1/{\underline{p}}}\r\|_{\lv}^{-1}
\sum_{j=1}^m\frac{{\lambda}_j|x_j+B_{l_j}|^{1/2}}{\|{\mathbf{1}}_{x_j+B_{l_j}}
\|_{\lv}}\noz\\
&\hs\hs\times\lf[\sum_{\ell\in\zz}
\int_{\{x\in\rn:\ (x,\ell)\in\widehat{x_j+B_{l_j}}\}}\lf|
\phi_{-\ell}\ast b_{j}^{(2)}(x)\r|^2\,dx\r]^{1/2}\noz\\
&\hs\ls \|b\|_{\mathcal{L}_{\vp,1,s,\underline{p}}^A({\rn})}.
\end{align}
We now estimate $b_j^{(3)}$. For this purpose, let $r\in(0,\underline{p})$
and $\vaz\in([2/r-1]\ln b/\ln\lz_-,\fz)$.
Then, for any $j\in\{1,\ldots,m\}$ and
$(x,\ell)\in \widehat{x_j+B_{l_j}}$, we have
\begin{align*}
\lf|\phi_{-\ell}\ast b_j^{(3)}(x)\r|
&\ls \int_{(x_j+B_{l_j+\omega})^{\com}}\frac{b
^{\vaz\ell\frac{\ln\lz_-}{\ln b}}}
{[b^{\ell}+\rho(x-y)]^{1+\vaz\frac{\ln\lz_-}{\ln b}}}
\lf|b(y)-P^s_{x_j+B_{l_j}}b(y)\r|\,dy\\
&\ls \frac{b^{\vaz\ell\frac{\ln\lz_-}{\ln b}}}{b^{\vaz l_j\frac{\ln\lz_-}{\ln b}}}\int_{(x_j+B_{l_j+\omega})^{\com}}\frac{b
^{\vaz l_j\frac{\ln\lz_-}{\ln b}}|b(y)-P^s_{x_j+B_{l_j}}b(y)|}
{b^{l_j(1+\vaz\frac{\ln\lz_-}{\ln b})}
+[\rho(x_j-y)]^{1+\vaz\frac{\ln\lz_-}{\ln b}}}
\,dy.\noz
\end{align*}
By this, Theorem \ref{2t2}, and the Tonelli theorem, we find that,
for any $m\in\nn$, $\{x_j+B_{l_j}\}_{j=1}^m\subset \BB$, and
$\{\lz_j\}_{j=1}^m\subset [0,\fz)$ with $\sum_{j=1}^m\lambda_j\neq0$,
\begin{align*}
&\lf\|\lf\{\sum_{i=1}^m
\lf[\frac{{\lambda}_i}{\|{\mathbf{1}}_{x_i+B_{l_i}}\|_{\lv}}\r]^{\underline{p}}
{\mathbf{1}}_{x_i+B_{l_i}}\r\}^{1/{\underline{p}}}\r\|_{\lv}^{-1}
\sum_{j=1}^m\frac{{\lambda}_j|x_j+B_{l_j}|^{1/2}}{\|{\mathbf{1}}_{x_j+B_{l_j}}
\|_{\lv}}\\
&\hs\hs\times\lf[\sum_{\ell\in\zz}
\int_{\{x\in\rn:\ (x,\ell)\in\widehat{x_j+B_{l_j}}\}}\lf|
\phi_{-\ell}\ast b_{j}^{(3)}(x)\r|^2\,dx\r]^{1/2}\\
&\hs\ls\lf\|\lf\{\sum_{i=1}^m
\lf[\frac{{\lambda}_i}{\|{\mathbf{1}}_{x_i+B_{l_i}}\|_{\lv}}\r]^{\underline{p}}
{\mathbf{1}}_{x_i+B_{l_i}}\r\}^{1/{\underline{p}}}\r\|_{\lv}^{-1}
\sum_{j=1}^m\frac{{\lambda}_j|x_j+B_{l_j}|}{\|{\mathbf{1}}_{x_j+B_{l_j}}
\|_{\lv}}\\
&\hs\hs\times\sum_{\ell=-\fz}^{l_j}{b^{-(l_j-\ell)\vaz\frac{\ln\lz_-}{\ln b}}}
\int_{(x_j+B_{l_j+\omega})^{\com}}\frac{b
^{\vaz l_j\frac{\ln\lz_-}{\ln b}}|b(y)-P^s_{x_j+B_{l_j}}b(y)|}
{b^{l_j(1+\vaz\frac{\ln\lz_-}{\ln b})}
+[\rho(x_j-y)]^{1+\vaz\frac{\ln\lz_-}{\ln b}}}
\,dy\\
&\hs\ls\|b\|_{\mathcal{L}_{\vp,1,s,\underline{p}}^{A,\vaz}(\rn)}
\sim\|b\|_{\mathcal{L}_{\vp,1,s,\underline{p}}^{A}(\rn)}.\noz
\end{align*}
Combining this, \eqref{4e3}, \eqref{4e4}, and \eqref{4e6}, we
conclude that
\begin{align*}
&\lf\|\lf\{\sum_{i=1}^m
\lf[\frac{{\lambda}_i}{\|{\mathbf{1}}_{x_i+B_{l_i}}\|_{\lv}}\r]^{\underline{p}}
{\mathbf{1}}_{x_i+B_{l_i}}\r\}^{1/{\underline{p}}}\r\|_{\lv}^{-1}
\sum_{j=1}^m\frac{{\lambda}_j|x_j+B_{l_j}|^{1/2}}{\|{\mathbf{1}}_{x_j+B_{l_j}}
\|_{\lv}}\\
&\hs\hs\times\lf[\sum_{\ell\in\zz}
\int_{\{x\in\rn:\ (x,\ell)\in\widehat{x_j+B_{l_j}}\}}\lf|
\phi_{-\ell}\ast b(x)\r|^2\,dx\r]^{1/2}\noz\\
&\hs\ls \|b\|_{\mathcal{L}_{\vp,1,s,\underline{p}}^A({\rn})},\noz
\end{align*}
which implies that,
for any $(x,k)\in\rn\times\zz$, $d\mu(x,k):=\sum_{\ell\in\zz}
|\phi_{-\ell}\ast b(x)|^2\,dx\,\delta_{\ell}(k)$ is a
$\vp$-Carleson measure on $\rn\times\zz$ and
$\|d\mu\|_{\mathcal{C}_{\vp,A}}
\ls\|b\|_{\mathcal{L}_{\vp,1,s,\underline{p}}^A({\rn})}$.
This finishes the proof of (i).

We now prove (ii). Indeed, let $f\in H_{A,{\rm fin}}^{\vp,\fz,s}(\rn)$
with the norm greater than zero. Then $f\in L^{\fz}(\rn)$ with compact support. Therefore, by the fact that
$b\in L^2_{\rm loc}(\rn)$ and \cite[(2.10)]{fl},  we know that
\begin{equation}\label{4e9}
\int_{\rn}f(x)\overline{b(x)}\,dx
\sim \sum_{\ell\in\zz}\int_{\rn}\varphi_{-\ell}\ast f(x)\overline{\phi_{-\ell}\ast b(x)}\,dx,
\end{equation}
where $\varphi\in\cs(\rn)$ satisfies that $\supp \widehat{\varphi}$
is compact and away from the origin and, for any $\xi\in\rn\setminus\{\mathbf{0}\}$,
$\sum_{k\in\zz} \widehat{\varphi}((A^\ast)^k\xi)\widehat{\phi}((A^\ast)^k\xi)=1$
with $A^\ast$ as the adjoint matrix of $A$.
Furthermore, from the Lusin area function
characterization of $\vh$ (see \cite[Theorem 4.4(i)]{lhy})
and the fact that $f\in \vh$, it follows that
\begin{align*}
\lf\|\varphi_{-\ell}\ast f\r\|_{T_A^{\vp}(\rn\times\zz)}\sim\|f\|_{\vh}<\fz.
\end{align*}
This, together with Lemma \ref{tent},
further implies that there exist $\{\lz_j\}_{j\in\nn}\subset [0,\fz)$
and a sequence $\{A_j\}_{j\in\nn}$ of anisotropic $(\vp,\fz)$-atoms supported, respectively, in
$\{\widehat{B^{(j)}}\}_{j\in\nn}$ with $\{B^{(j)}\}_{j\in\nn}\subset \BB$
such that, for almost every $(x,\ell)\in\rn\times\zz$,
$$\varphi_{-\ell}\ast f(x)=\sum_{j\in\nn}\lz_jA_j(x,\ell)$$
and
\begin{align*}
0<\lf\|\lf\{\sum_{j\in\nn}\lf[\frac{\lambda_j}
{\|\mathbf1_{B^{(j)}}\|_{\lv}}\r]^{\underline{p}}
\mathbf1_{B^{(j)}}\r\}^{1/\underline{p}}\r\|_{\lv}
\ls\|f\|_{\vh}.
\end{align*}
Combining this, \eqref{4e9}, the H\"{o}lder inequality,
the Tonelli theorem, and the size condition
of $A_j$, we conclude that, for any
$f\in H_{A,{\rm fin}}^{\vp,\fz,s}(\rn)$,
\begin{align*}
&\lf|\int_{\rn}f(x)\overline{b(x)}\,dx\r|\\
&\hs\ls \sum_{\ell\in\zz}\sum_{j\in\nn}\lz_j\int_{\rn}\lf|A_j(x,\ell)\r|
\lf|\phi_{-\ell}\ast b(x)\r|\,dx\noz\\
&\hs\ls\sum_{j\in\nn}\lz_j\lf[\sum_{\ell\in\zz}\int_{\{x\in\rn:\
(x,\ell)\in\widehat{B^{(j)}}\}}\lf|A_j(x,\ell)\r|^2
\,dx\r]^{\frac12}
\lf[\sum_{\ell\in\zz}\int_{\{x\in\rn:\
(x,\ell)\in\widehat{B^{(j)}}\}}
\lf|\phi_{-\ell}\ast b(x)\r|^2 \,dx\r]^{\frac12}\noz\\
&\hs\ls \sum_{j\in\nn}\lz_j\lf\|A_j\r\|_{T^2_2(\rn\times\zz)}
\lf[\sum_{\ell\in\zz}\int_{\{x\in\rn:\
(x,\ell)\in\widehat{B^{(j)}}\}}\lf|\phi_{-\ell}\ast b(x)\r|^2\,dx\r]^{1/2}\noz\\
&\hs\ls \sum_{j\in\nn}\frac{\lz_j|B^{(j)}|^{1/2}}{\|\mathbf{1}_{B^{(j)}}\|_{\lv}}
\lf[\sum_{\ell\in\zz}\int_{\{x\in\rn:\
(x,\ell)\in\widehat{B^{(j)}}\}}\lf|\phi_{-\ell}\ast b(x)\r|^2
\,dx\r]^{1/2}\noz\\
&\hs\ls \|f\|_{\vh}\widetilde{\|d\mu\|}_{\mathcal{C}_{\vp,A}}.
\end{align*}
From this, Theorem \ref{2t1}, $p_+\in(0,2)$, Corollary \ref{2c1},
and Remark \ref{4r1}, we further conclude that
\begin{align*}\label{5e3}
\|b\|_{\mathcal{L}_{\vp,1,s,\underline{p}}^A({\rn})}
\ls\|d\mu\|_{\mathcal{C}_{\vp,A}},
\end{align*}
which completes the proof of (ii) and hence of Theorem \ref{2t3}.
\end{proof}

\begin{remark}
Let $\vp\in C^{\log}(\rn)$ with $p_+\in(0,1]$, and $s$ be as in \eqref{4e1}.
Then, using Proposition \ref{2p2},
we know that Theorem \ref{2t3} also gives the Carleson
measure characterization of the anisotropic variable
Campanato space $\mathcal{L}_{\vp,1,s}^A({\rn})$.
We should point out that, even in this case, Theorem \ref{2t3} is also new.
\end{remark}

%
%


\bigskip

\noindent Long Huang and Xiaofeng Wang

\medskip

\noindent School of Mathematics and Information Science,
Key Laboratory of Mathematics and Interdisciplinary Sciences of the Guangdong Higher Education Institute,
Guangzhou University, Guangzhou, 510006, People's Republic of China

\smallskip

\noindent {\it E-mails}:
\texttt{longhuang@gzhu.edu.cn} (L. Huang)

\noindent\phantom{{\it E-mails:}}
\texttt{wxf@gzhu.edu.cn} (X. Wang)

\end{document}